\documentclass[11pt, oneside, reqno]{amsart}
\usepackage{amsmath,amsthm,amssymb,mathtools,tocvsec2,pdflscape,cite}
\usepackage{latexsym}
\usepackage{dsfont}
\usepackage{color,soul}
\usepackage[utf8]{inputenc}
\usepackage[T1]{fontenc}
\usepackage[mathscr]{euscript}
\usepackage[tmargin=1in,bmargin=1in,rmargin=1in,lmargin=1in]{geometry}
\usepackage[hidelinks]{hyperref}
\usepackage[shortlabels]{enumitem}
\usepackage{tikz-cd} 
\usepackage{bbm}
\usepackage{MnSymbol}

\allowdisplaybreaks
\numberwithin{equation}{section}

\newtheorem{theorem}{Theorem}[section]
\newtheorem{lemma}[theorem]{Lemma}

\theoremstyle{definition}
\newtheorem{remark}[theorem]{Remark}

\newcommand{\beas}{\begin{eqnarray*}}
\newcommand{\eeas}{\end{eqnarray*}}
\newcommand{\bes} {\begin{equation*}}
\newcommand{\ees} {\end{equation*}}
\newcommand{\be} {\begin{equation}}
\newcommand{\ee} {\end{equation}}
\newcommand{\bea} {\begin{eqnarray}}
\newcommand{\eea} {\end{eqnarray}}
\newcommand{\beals} {\begin{align*}}
\newcommand{\eeals} {\end{align*}}
\newcommand{\beal} {\begin{align}}
\newcommand{\eeal} {\end{align}}

\newcommand{\hol}{\mathcal {O}}

\newcommand{\cont}{\mathcal C}

\newcommand{\C}{\mathbb C}
\newcommand{\Cn}{{\mathbb{C}^n}}
\newcommand{\N}{\mathbb N}
\newcommand{\Z}{\mathbb Z}
\newcommand{\R}{\mathbb{R}}
\newcommand{\Rea}{\operatorname{Re}}
\newcommand{\Ima}{\operatorname{Im}}

\newcommand{\F}{\mathcal F}

\newcommand{\rea}{\operatorname{Re}}

\newcommand{\ltt}{\left(}
\newcommand{\rtt}{\right)}

\newcommand{\Om}{\Omega}

\newcommand{\eps}{\epsilon}

\begin{document}
\begin{abstract}
The Fantappiè and Laplace transforms realize isomorphisms between analytic functionals supported on a convex compact set $K\subset\Cn$ and certain spaces of holomorphic functions associated with $K$. Viewing the Bergman space of a bounded domain in $\Cn$ as a subspace of the space of analytic functionals supported on its closure, the images of the restrictions of these transforms have been studied in the planar setting. For the Fantappi{\`e} transform, this was done for simply connected domains (Napalkov Jr--Yulmukhametov, 1995), and for the Laplace transform, this was done for convex domains (Napalkov Jr--Yulmukhametov, 2004). In this paper, we study this problem in higher dimensions for strongly convex domains, and establish duality results analogous to the planar case. We also produce examples to show that the planar results cannot be generalized to all convex domains in higher dimensions. 
    
\end{abstract}
\title{Dual realizations of Bergman spaces on strongly convex domains}
\author{Agniva Chatterjee}
\address{Department of Mathematics, Indian Institute of Technology Palakkad, Palakkad, Kerala-678623, India}
\email{agnivac@iitpkd.ac.in}
\maketitle
\section{Introduction}
For a compact convex subset $K$ of $\C^n$, the following spaces are known to be isomorphic:
\begin{itemize}
    \item [$(a)$] $\hol'(K)$, the space of analytic functionals on $K$, 
    \item [$(b)$] $\hol(K^*)$, the space of holomorphic functions on $K^*$, the dual complement of $K$, given by 
    \bes 
        K^*=\left\{z\in\C^n: \left\langle\zeta,z\right\rangle\neq 1,\,\forall \zeta\in K\right\},
    \ees 
    where $\left\langle\zeta,z\right\rangle=\zeta_1z_1+\zeta_2z_2+\cdots\zeta_nz_n$, and 
   \label{def:exp fn} \item[$(c)$] $\hol_{\operatorname{\operatorname{exp}}}(K)$, the space of entire functions $F$ on $\C^n$ such that, for every $\eps>0$, there exists a $C_\eps>0$ so that 
    \bes
    |F(z)|\leq C_\eps e^{(1+\eps)H_K(z)},\quad z\in \C^n,
    \ees
    where $H_K$ is defined as
    \bea\label{eq:supprt def}
H_K(z)=\sup\limits_{\zeta\in K}\rea\langle\zeta,z\rangle,\,z\in\C^n.
\eea
Note that $H_K(\overline z)$ coincides with the classical support function of $K$.
\end{itemize}

The duality between $(a)$, and $(b)$ was established by Aizenberg and Martineau in the case of convex sets, and by Gindikin and Henkin in the more general case of $\C$-convex sets; see \cite{Ab66,MA67,GsHg78}. The  {\em Fantappi{\`e} transform} on $\hol'(K)$, given by
\bes
\mathcal{F}_k\ltt\mu\rtt(z)=\mu\ltt\frac{1}{\ltt1-\left\langle\cdot,z\right\rangle\rtt^k}\rtt,\quad \mu\in\hol'(K), z\in K^*,
\ees
for any fixed integer $k\geq 1$,  gives an explicit isomorphism between the two spaces (see Znamenskii \cite{Zs85}). Since $\hol'\ltt K\rtt$ is a large space, it is of interest to study the restrictions of $\mathcal F_k$ to certain important subclasses of $\hol'\ltt K\rtt$. For instance, when $K$ is the closure of a bounded domain $\Om\subset\Cn$, then under different assumptions on $\Om$, the restriction of $\mathcal F_{1}$ to the Bergman space of $\Om$, the restriction of $\F_n$ to the holomorphic Hardy space of $\Om$, and the restriction of $\F_n$ to the space of holomorphic functions with polynomial growth on $\Om$, has been considered in \cite{MP05}, \cite{Li02}, and \cite{AaKl13}, respectively. 

The Paley-Wiener type duality between $(a)$ and $(c)$ was established by P{\'o}lya-Ehrenpreis-Martineau, via the Laplace transform 
\[
\mathcal{L}(\mu)(z)=\mu\ltt e^{\langle\cdot,z\rangle}\rtt,\quad \mu\in\hol'(K), z\in \C^n.
\] See \cite[Theorem 4.1.9]{APS}. As is the case for $\F_k$, the restriction of $\mathcal L$ to different subspaces of $\hol'(K)$ has been considered. For instance, the restriction of $\mathcal L$ to holomorphic Hardy spaces of a bounded domain has been studied for planar domains, strongly convex domains in $\C^n$, and some weakly convex domains in $\C^2$ in \cite{LuYl91}, \cite{Li02}, and \cite{AC24}, respectively. 

In this note, we consider the restriction of the Fantappi{\`e} and the Laplace transforms on $p$-Bergman spaces on certain bounded domains $\Om\subset\C^n$.
For a bounded domain $\Om\subset\C^n$, and $p\in(1,\infty)$, the {\em $p$-Bergman space on $\Om$} is defined as 
\bea\label{eq:def berg}
A^p\ltt\Om\rtt=\left\{f\in \hol\ltt\Om\rtt:\|f\|_{L^p(\Om)}=\ltt\int_\Om\left|f(\zeta)\right|^p dV(\zeta)\rtt^{1/p}<\infty\right\},
\eea
where $V$ is the Lebesgue measure on $\Om$. Each $A^p(\Om)$ is a Banach space when endowed with the $L^p$-norm, and can be viewed as a subset of $\hol'\ltt\overline\Om\rtt$ via the (possibly non-injective) inclusion
\be\label{eq:inclusion}
    A^p(\Om)\overset{\iota}{\hookrightarrow} L^p(\Om)\xrightarrow{\theta} \hol'\ltt\overline\Om\rtt,
\ee
where $\theta:L^p(\Om)\rightarrow\hol'\ltt\overline\Om\rtt$ is given by
\be\label{eq: theta}
\theta:f\mapsto \left(\mu_f:g\mapsto\frac{n!}{\pi^n}\int_\Om \overline{f(\zeta)} g(\zeta) dV(\zeta)\right).
\ee
The composition $\mathcal F_k\circ\theta\circ\iota$, denoted simply by $\mathcal F
_k$ for convenience, is referred to as the Fantappi{\`e} transform on the $p$-Bergman space, and is explicitly given by
\bea\label{eq:Fan berg}
\mathcal{F}_{k}(g)(z)=\frac{n!}{\pi^n}\int_{\Om} \frac{\overline{ g(\zeta)}}{\ltt 1-\langle \zeta,z\rangle\rtt^{k}}dV(\zeta),\quad g\in A^p(\Om),\, z\in \text{int}\,\Om^*=\overline\Om^*.
\eea

First, we study the question of when $\mathcal F_k$ is a normed space isomorphism between $A^p(\Om)$ and $A^p\left(\Om^*\right)$, where, in an abuse of notation, we denote the open set $\overline\Om^*$ by $\Om^*$. The range of $\mathcal F_k$ can be sensitive to the choice of $k$, and we focus on the case $k=n+1$. This is motivated by the following observation: if $\mathbb B^n$ denotes the unit ball in $\Cn$, then
\bes
\mathcal F_{n+1}\ltt f\rtt(z)=\overline{f(\overline{z})},\quad f\in A^2\ltt\mathbb B^n\rtt, z\in \mathbb B^n, 
\ees
and thus, $\mathcal{F}_{n+1}$ is an isometry between $A^2(\mathbb B^n)$ and $A^2\ltt\mathbb B^{n^*}\rtt$. This follows from the fact that $\frac{n!}{\pi^n}\ltt1-\left\langle\overline{\zeta},z\right\rangle\rtt^{-(n+1)}$, $\ltt\zeta,z\rtt\in \mathbb B^n\times\mathbb B^n$, is the reproducing kernel of $A^2(\mathbb B^n)$, and $\overline{\mathbb B^n}^*=\mathbb B^n$. {Note that $\mathbb B^n$ is a strongly convex {\em circled} domain. Recall that a domain $\Om\subseteq\C^n$, is said to be circled, if $e^{i\theta}z \in \Omega$, for all $z \in \Omega$ and $\theta \in [0, 2\pi)$. The above observation generalizes to the class of bounded strongly convex circled domains. In fact, we prove the following theorem in the general setting of any bounded strongly $\C$-convex domain with a strongly convex dual complement.}
\begin{theorem}\label{th:main1}
    Let $\Om\subset\C^n$ be a bounded $\cont^2$-smooth strongly $\C$-convex domain containing the origin, such that $\Om^*$ is strongly convex. Then, for any $p\in(1,\infty)$, $\mathcal{F}_{n+1}$ is a normed space isomorphism between $A^p\ltt\Om\rtt$ and $A^p\ltt\Om^*\rtt$. 
\end{theorem}
We are not aware of any characterization of the class of bounded $\cont^2$-smooth strongly $\C$-convex domains whose dual complements are strongly convex. {However, it is easy to check that the dual complement of a bounded $\cont^2$-smooth circled strongly convex (and, therefore, strongly $\C$-convex) domain is always strongly convex. For examples of non-circled strongly $\C$-convex (but not strongly convex) domains with strongly convex dual complements; see Remark~\ref{rk:thm 1 2}. In general, the dual complement of a bounded $\cont^2$-smooth strongly convex (hence, strongly $\C$-convex) domain is not necessarily strongly convex; see Lemma~\ref{le:counter ellip}. 
We do not know if strong convexity of the dual complement is a necessary condition for the duality in Theorem~\ref{th:main1} to hold; see Remark~\ref{rk:ques L^p bdd} for more details.

  Theorem~\ref{th:main1} can now be used to describe the range of the Laplace transform on $A^2(\Om)$, given by
    \bea\label{eq:lap berg}
    \mathcal L(f)(z)=\int_\Om \overline{f(\zeta)} e^{\left\langle\zeta,z\right\rangle} dV(\zeta),\quad z\in\C^n,
    \eea
    under an additional assumption of convexity on $\Om$.
\begin{theorem}\label{th:PW bergman}
    Let $\Om\subset\C^n$ be a bounded $\cont^2$-smooth strongly convex domain containing the origin, such that $\Om^*$ is strongly convex. Then the Laplace transform $\mathcal L$ on $A^2(\Om)$
    is a normed space isomorphism between $A^2(\Om)$, and the weighted Bergman space
    \[    A^2(\C^n,\omega_\Om)=\left\{F\in\hol(\C^n):\Vert F\Vert_{\omega_\Om}^2=\int_{\C^n}|F(z)|^2 \omega_\Om(z)<\infty\right\},
    \]
  with norm $\|\cdot\|_{\omega_\Om}$, where 
  \beas
  \omega_\Omega(z)=e^{-2H_\Om(z)}\|z\|^{n+\frac{1}{2}} \ltt dd^cH_\Om\rtt^n(z),\quad z\in\C^n,
  \eeas
  and $H_\Om(\overline z)$ is the support function of $\Om$.
\end{theorem}

Similar to Theorem~\ref{th:main1}, the hypothesis of this theorem holds true for any bounded $\cont^2$-smooth strongly convex circled domain.

{
Furthermore, when $\Om\subset\C^n$ is a bounded strongly convex domain, the weighted Bergman space $A^2(\C^n,\omega_\Om)$ is isomorphic to $A^2(\C^n,\mu_\Om)$ via the identity map, where
\[
\mu_\Om(z)=\|e^{\langle\cdot,z\rangle}\|^{-2}_{L^2(\Om)} \ltt dd^cH_\Om\rtt^n(z);
\]
see Lemma~\ref{le:comp berg sp}.
In the planar case, similar results as Theorem~\ref{th:main1} and Theorem~\ref{th:PW bergman} are known to hold for $p=2$, under a much less restrictive hypothesis on the domain. In particular, for {\em every} bounded convex domain $\Om\subset\C$, it is known that $\mathcal F_2$ is a normed space isomorphism between $A^2(\Om)$ and $A^2(\Om^*)$, and $\mathcal L$ is a normed space isomorphism between $A^2(\Om)$ and $A^2(\C,\mu_\Om)$; see \cite{NY95,MK99,IY04}.  However, we show that these results cannot be extended to higher dimensions in their full generality. 
\begin{theorem}\label{th:counter}
Let 
$\Om=\left\{\ltt \zeta_1,\zeta_2\rtt\in\C^2: |\zeta_1|+|\zeta_2|<1\right\}$. Then, 
\begin{itemize}
    \item [$(i)$] $\mathcal F_3$ is a bounded injective operator from $A^2(\Om)$ onto a proper subspace of $A^2(\Om^*)$,
    \item [$(ii)$] $\mathcal L$ is a bounded injective operator from $A^2(\Om)$ onto a proper subspace of both $A^2(\C^2,\omega_{\Om})$, and $A^2(\C^2,\mu_\Om)$. Moreover,
    \[
    \mathcal L(A^2(\Om))\subsetneq A^2(\C^2,\omega_\Om)\subsetneq A^2(\C^2,\mu_\Om).
    \]
    \end{itemize}
\end{theorem}
}
We provide an outline of the proofs. Let $\Om$ be as in Theorem~\ref{th:main1}. Then, it admits a strongly $\C$-convex neighborhood basis, and hence, $\overline\Om$ is $\C$-convex; see {\cite[page 27]{APS}}. Thus, by the duality result stated earlier, $\F_{n+1}\ltt\hol'\ltt\overline\Om\rtt\rtt=\hol\ltt\Om^*\rtt$. In particular, $\F_{n+1}$ maps $A^p(\Om)$ to $\hol\ltt\Om^*\rtt$. Thus, to prove Theorem~\ref{th:main1}, it suffices to establish that, for each $p\in(1,\infty)$, 
\begin{itemize}
    \item [$1.$] $\mathcal F_{n+1}$ is an $L^p(\Om)$-$L^p(\Om^*)$-bounded operator,
    \item [$2.$] $\mathcal F_{n+1}$ is injective on $A^p(\Om)$, and
    \item [$3.$] $\mathcal F_{n+1}$ maps $A^p(\Om)$ onto $A^p(\Om^*)$.    
\end{itemize}
For $1$, we use the domain's strong $\C$-convexity to convert  $\mathcal F_{n+1}$ into a singular integral operator on $\Om^*$. This operator is then compared with a projection operator, $\mathcal B_{\Om^*}$, on $L^p(\Om^*)$. This projection operator is a simpler version of an operator used in the literature to study the regularity properties of the Bergman projection on strictly pseudoconvex domains via a Kerzman--Stein-type approach; see \cite{LI84}, \cite[\S~7.1]{Ra98}, and \cite{LS12}. Lemma~\ref{le:bdd of LS ker} is the key boundedness property of $\mathcal B_{\Om^*}$ that we need, and this result may be of independent interest. For both $2$ and $3$, we use that $A^p(\Om)'
$ is isomorphic to $A^q(\Om)$ when $p^{-1}+q^{-1}=1$; see Lemma~\ref{le:dual of Ap}. This duality allows us to show that $\theta\circ\iota$ in \eqref{eq:inclusion} is injective, following which, the injectivity of $\F_{n+1}$ on $\hol'\ltt\overline\Om\rtt$ gives $2$. For $3$, we show that each function $g$ in $A^p(\Om^*)$ induces a bounded functional on $A^q(\Om)$, and thus, can be mapped to a function $\phi_g$ in $A^p(\Om)=A^q(\Om)'$. Using a Cauchy--Fanttapi{\`e}-type reproducing formula for $A^p(\Om)$ --- see Lemma~\ref{le:repr} --- we show that $\mathcal F_{n+1}(\phi_g)$ is, in fact, $g$.

To prove Theorem~\ref{th:PW bergman}, we relate $\F_{n+1}$ with $\mathcal L$ via the following commutative diagram:
\bes \begin{tikzcd}
 A^2(\Om) \arrow{dr}{\mathcal L}  \arrow{r}{\mathcal F_{n+1}} & A^2(\Om^*)  
 \\%
& A^2(\Cn,\omega_\Om) \arrow[swap]{u}{\mathfrak B_n}
&  
\end{tikzcd}
\ees
where $\mathfrak B_n:F\mapsto\ltt z\mapsto \int_0^\infty F(tz)t^ne^{-t}dt\rtt$ is known as the Borel transform. Owing to Theorem~\ref{th:main1}, it suffices to show that $\mathfrak B_n$ is a normed space isomorphism between $A^2(\Cn,\omega_\Om)$ and $A^2(\Om^*)$.

Let $\Om$ be as in Theorem~\ref{th:counter}. Then, both $\Om$ and $\Om^*$ are complete Reinhardt domains (in fact, $\Om^*$ is the unit polydisc in $\C^2$). This allows us to give characterizations of Bergman-space functions on both domains using power-series representations. We produce an explicit element in $A^2(\Om^*)$ that is not in the image of $\mathcal F_3$, and an element in $A^2(\C^2,\omega_\Om)$, {and hence in $A^2(\C^2,\mu_\Om)$} which is not in the image of $\mathcal L$.
\vspace{2pt}
\medskip

\noindent \textbf{Acknowledgements.} I am deeply grateful to my thesis advisor, Purvi Gupta, for her insightful guidance, many helpful ideas, and for engaging in regular discussions throughout the course of this project. I would also like to thank her for helping me with the writing of this paper. I also thank Mihai Putinar for valuable discussions and for suggesting that the Fantappi{\`e} kernel with exponent $n+1$ might be effective in our context, as well as for highlighting connections of the Fantappi{\`e} transform to various other topics that may lead to promising future directions. Finally, I am grateful to the anonymous referee for providing many constructive comments that significantly improved this paper.

This work is supported by a scholarship from the Indian Institute of Science, and the DST-FIST
programme (grant no. DST FIST-2021 [TPN-700661]).

\section{Preliminaries}
In this section, we collect some preparatory results and observations for our main objects of study. 
\subsection{Notation} The following notation will be used throughout the paper. 
\begin{itemize}
    \item[(1)] $\mathbb B^n$ denotes the unit Euclidean ball in $\C^n$.
     \item [(2)] $d=\partial+\overline{\partial}$ denotes the standard exterior derivative.
    \item [(3)] $d^c={i}\ltt{\overline{\partial}-\partial}\rtt.$
    \item [(4)] Given a bounded domain $\Om$, $\sigma_\Om$ denotes the Euclidean surface area measure on $b\Om$.
    \item [(5)] Given two $\R$-valued functions $f$ and $g$ on set a $X$,
    \begin{itemize}
   \item $g\lesssim f$ on $X$ denotes the existence of $C_1>0$ such that $C_1 g(x)\leq f(x)$, for all $x\in X$.
   \item $g\approx f$ on $X$ denotes the existence of $C_1,C_2>0$ such that $C_1 g(x)\leq f(x)\leq C_2 g(x)$, for all $x\in X$. 
    \end{itemize}
    \item[(6)] $\N$ denotes the set of all natural numbers, i.e., the set of all positive integers union $\{0\}$.
   {
    \item[(7)] Given a bounded $\cont^2$-smooth domain $\Om$ that is star-convex with respect to the origin, $m_\Om$ denotes the Minkowski functional of $\Om$, defined in \eqref{eq:mink def}.   
    \item[(8)] Given a bounded domain $\Om$, $H_\Om$ denotes the function as defined in \eqref{eq:supprt def}.
    \item[(9)] Given a bounded $\cont^2$-smooth convex domain $\Om$ containing the origin, $\rho_\Om$ denotes the defining function as given in \eqref{eq:defn fns}.
    \item[(10)] Given a domain $\Om$, and $p\in(1,\infty)$, $A^p(\Om)$ denotes the $p$-Bergman space of $\Om$, defined in \eqref{eq:def berg}. 
    \item[(11)] Given a domain $\Om$, $\hol_{exp}(\Om)$ denotes a subspace of exponential entire functions, defined in \eqref{eq:Borel dfn cond}. 
    \item[(12)] Given a $g\in A^p(\Om)$ and $k\in\Z_+$,
    $\mathcal F_k(g)$ denotes the Fantappi{\`e} transform of $g$, defined in \eqref{eq:Fan berg}.
    \item[(13)] Given a $f\in A^2(\Om)$,
    $\mathcal L(f)$ denotes the Laplace transform of $f$, defined in \eqref{eq:lap berg}.
    \item[(14)] Given a $F\in\hol_{exp}(\Om)$, $\mathfrak B_n(F)$ denotes the Borel transform of $F$, defined in \eqref{eq:def borel}.
    }
\end{itemize}
 Let $D\subset \C^n$ be a bounded $\cont^2$-smooth domain that is star-convex with respect to the origin. The Minkowski functional of $D$ is given by
\bea\label{eq:mink def}
m_D(z)=\inf\left\{t>0:t^{-1}z\in D\right\},\quad z\in\C^n.
\eea
Due to the star-convexity of $D$, $m_D$ is a positively $1$-homogeneous function on $\C^n$, i.e., $f(tz)=tf(z)$ for all $z\in\C^n$ and $t>0$. Moreover, $D=\{z:m_D(z)<1\}$, and $bD=\{z:m_D(z)=1\}$. We note the following result on the regularity of $m_D$.

\begin{lemma}\label{le:smth minkw}
    Let $D\subset\C^n$ be a $\cont^2$-smooth bounded domain that is star-convex with respect to the origin. Further, assume that $D$ has only non-radial tangents, i.e., for any $\zeta\in bD$, $\rea\langle\zeta,\overline{\eta}(\zeta)\rangle\neq 0$, where $\eta(\zeta)$ is a normal to $bD$ at $\zeta$. Then $m_D$ is $\cont^2$-smooth on $\C^n\setminus\{0\}$.
\end{lemma}
\begin{proof}
Since $bD$ is $\cont^2$-smooth, there exists a defining function of $D$, $\rho$, which is $\cont^2$-smooth in $\C^n$. We define $\phi:\C^n\times \ltt\R\setminus\{0\}\rtt\rightarrow\R$ as 
\[
\phi\ltt z,t\rtt=\rho\ltt t^{-1}z_1,t^{-1}z_2,\cdots,t^{-1}z_n\rtt,\quad \ltt z,t\rtt\in\C^n\times\ltt\R\setminus\{0\}\rtt.
\]
As $\rho\in \cont^2\ltt\C^n\rtt$, it follows that $\phi\in\cont^2\ltt\C^n\times\ltt\R\setminus\{0\}\rtt\rtt$. Also, since $m_D$ is positively $1$-homogeneous, for any $z\in\C^n\setminus\{0\}$, $\frac{z}{m_D(z)}\in bD$. Consequently, 
\[
\phi\ltt z,m_D(z)\rtt=0,\quad \forall z\in\C^n\setminus\{0\}.
\]
Let us now fix a $\tilde z\in\C^n\setminus\{0\}$. We claim that $\frac{\partial \phi}{\partial t}(\tilde z,m_D(\tilde z))\neq 0$. Suppose $\frac{\partial \phi}{\partial t}(\tilde z,m_D(\tilde z))=0$. Then, using the chain rule, we get that  
\beas
\frac{\partial \phi}{\partial t}(\tilde z,m_D(\tilde z))=\frac{-1}{m_D(\tilde z)}\Rea\left\langle2\partial \rho\ltt\frac{\tilde z}{m_D\ltt\tilde z\rtt}\rtt,\frac{\tilde z}{m_D\ltt\tilde z\rtt}\right\rangle=0.
\eeas
As $m_D\ltt\tilde z\rtt\neq 0$, it must be that
$
\Rea\left\langle2\partial \rho\ltt\frac{\tilde z}{m_D\ltt\tilde z\rtt}\rtt,\frac{\tilde z}{m_D\ltt\tilde z\rtt}\right\rangle=0$. This is a contradiction, since $D$ has only non-radial tangents by assumption. Thus, we may apply the implicit function theorem to $\phi$ and conclude that $m_D$ is a $\cont^2$-smooth function on $\C^n\setminus\{0\}$.   
\end{proof}

The next result allows us to freely move between $D$ and $D^*$ under the assumption of strong $\C$-convexity.

\begin{lemma}\label{le:cov dom}
Let $D\subset\C^n$ be a bounded $\cont^2$-smooth strongly $\C$-convex domain that is star-convex with respect to the origin. Further, assume that $D$ has only non-radial tangents, i.e., for any $\zeta\in bD$, $\rea\langle\zeta,\overline\eta(\zeta)\rangle\neq 0$, where $\eta(\zeta)$ is a normal to $bD$ at $\zeta$. Then, the map $T_{D}:D\rightarrow D^*$, defined as
\bea\label{eq:cov map}
T_{D}\ltt\zeta\rtt=\begin{cases}
m_{D}^2(\zeta)\dfrac{\partial m_{D}(\zeta)}{\left\langle\partial m_{D}(\zeta),\zeta\right\rangle},&\quad \zeta\in D\setminus\{0\},\\
0,&\quad \zeta=0,
\end{cases}
\eea
\begin{itemize}
    \item [(a)] is invertible, with $T_{D^*}$ as the inverse map,
    \item [(b)] is a homeomorphism between $D$ and $D^*$, and
    \item [(c)] is a $\cont^1$-diffeomorphism between $D\setminus\{0\}$ and $D^*\setminus\{0\}$.
\end{itemize}
Moreover, if $f$ is a measurable function on $D^*$, then
\bea\label{eq:cov}
\int_{D^*} f\ltt\zeta\rtt dV(\zeta)=\int_{D} f\ltt T_{D}(\zeta)\rtt h_{D}(\zeta) dV(\zeta),
\eea
where $h_{D}$ is a positive function on $D$, which is bounded above and below on $D$ by positive constants.
    \end{lemma}
    \begin{proof}
        Let $S_{D}:bD\rightarrow bD^*$ be given by 
            \bes            S_{D}\ltt\eta\rtt=\frac{\partial m_{D}(\eta)}{\left\langle\partial m_{D}(\eta),\eta\right\rangle},\qquad \eta\in bD.
            \ees
        The following implications follow from \cite{Li02}.
        \begin{itemize}
        \item [(i)] $D^*$ is also a  bounded $\cont^2$-smooth strongly $\C$-convex domain that is star-convex with respect to the origin, and $m_{D^*}$ is $\cont^1$-smooth on $\C^n\setminus\{0\}$; see \cite[Lemma~$9$]{Li02}. As $m_{D^*}$ is positively $1$-homogeneous, by Euler's identity, for any $\xi\in bD^*$,
\bes\label{eq:euler}
\rea\left <\partial m_{D^*}(\xi),\xi\right>= \frac{m_{D^*}(\xi)}{2}= \frac{1}{2}.
\ees  Hence, $D^*$ has only non-radial tangents.  
        \item[(ii)] $S_{D}$ is a $\cont^1$-diffeomorphism from $bD$ onto $bD^*$, whose inverse is $S_{D^*}$; see \cite[page~ $295-296$]{Li02}.
            \item[(iii)] If $g$ is a measurable function with respect to $\sigma_{D^*}$, then,
            \[
            \int_{bD^*} g\ltt\xi\rtt d\sigma_{D^*}(\xi)=\int_{bD} g\ltt S_{D}(\eta)\rtt \widetilde{h}(\eta) d\sigma_{D}(\eta),
            \]
            for some positive continuous function $\widetilde{h}$ on $bD$; see \cite[Lemma~$25$]{Li02}.
            \item [(iv)] If $f$ is a Lebesgue integrable function on $D^*$, then 
\bea\label{eq:polar coordinate}
\int_{D^*} f\ltt \zeta\rtt dV(\zeta)&=&\int_0^1\int_{bD}f\ltt rS_{D}(\eta)\rtt r^{2n-1} {j_D(\eta)} dr\,  d\sigma_{D}(\eta)\nonumber\\
&=&\int_0^1\int_{bD^*}f\ltt r\xi\rtt r^{2n-1} {\ell_{D^*}(\xi)} dr\,  d\sigma_{D^*}(\xi),
\eea
for some positive continuous functions $j_D$ and $\ell_{D^*}$ on $bD$ and $bD^*$, respectively; see \cite[page 308]{Li02}.
        \end{itemize}
    We write an arbitrary $\zeta\in D$ as $\zeta=r\eta$, where $r\in(0,1)$, and $\eta\in bD$. Then using the $1$-homogenity of $m_{D}$, it follows that $T_{D}(\zeta)=rS_{D}\ltt\eta\rtt$. From (ii), it is immediate that $T_{D}$ is a homeomorphism between $D$ and $D^*$. Furthermore, since $m_{D}$ is $\cont^2$-smooth on $\C^n\setminus\{0\}$, $T_{D}$ is a $\cont^1$-diffeomorphism between $D\setminus\{0\}$ and $D^*\setminus\{0\}$. 
    Finally, using (iv), we get that
    \beas
    \int_{D^*} f\ltt \zeta\rtt dV(\zeta)
    &=&\int_0^1\int_{bD}f\ltt rS_D\ltt\eta\rtt\rtt r^{2n-1} {j_D(\eta)} d\sigma_{D}(\eta)\,dr\\
&=&\int_0^1\int_{bD}f\ltt T_{D}(r\eta)\rtt r^{2n-1} {j_D(\eta)} d\sigma_{D}(\eta)\,dr\\
    &=&\int_0^1\int_{bD}(f\circ T_{D})(r\eta) r^{2n-1} {h_{D}(r\eta)} \ell_D(\eta)d\sigma_{D}(\eta)\,dr\\
     &=&\int_{D} (f\circ T_{D})(\zeta) h_{D}(\zeta) dV(\zeta),
    \eeas
    where $h_{D}(r\eta)=j_D(\eta){\ell_D(\eta)}^{-1}$, for $r\in(0,1), \eta\in bD$. We obtained the last equality by applying (iv) to $D$, instead of $D^*$. This concludes the proof.
    \end{proof}

Note that if the domain $\Om\subset\C^n$ satisfies the hypothesis of Theorem~\ref{th:main1}, then $\Om^*$ satisfies the hypothesis of Lemma~\ref{le:cov dom} due to strong convexity, and hence $(\Om^*)^*=\Om$ also satisfies the hypothesis of Lemma~\ref{le:cov dom}.

Now, let 
\be\label{eq:defn fns}
\rho_D(\zeta)= m_D^2\ltt\zeta\rtt-1,\quad\zeta\in\C^n.
\ee
Since $m_D$ is a $C^2$-smooth function on $\C^n\setminus\{0\}$, $\rho_D$ is a $C^2$-smooth defining function of $D$. In fact, the following lemma shows that $\rho_D$ is also strongly convex on $\C^n\setminus\{0\}$.

\begin{lemma}\label{le:str cvx of Minkwski}
Let $D\subset\C^n$ be a bounded $\cont^2$-smooth strongly convex domain containing the origin. Then $\rho_D(\zeta)=m_D^2(\zeta)-1$ is a strongly convex function on $\C^n\setminus\{0\}$, i.e.,
$\operatorname{Hess}(\rho_D)(\zeta)$ is positive definite, for all $\zeta\in\C^n\setminus\{0\}$.
\end{lemma}
\begin{proof}
Let $\ltt \zeta,w\rtt\in \ltt\C^n\setminus\{0\}\rtt^2$. Writing $w$ in real co-ordinates as $w=\ltt u_1,u_2,\cdots,u_{2n}\rtt$, we get that,
\beas
w^T \cdot \operatorname{Hess}(\rho_D)(\zeta)\cdot w &=& 
\sum_{j,k=1}^{2n} \frac{\partial^2\rho_D}{\partial x_j \partial x_k}(\zeta) u_ju_k\\
&=&2\sum_{j,k=1}^{2n} \ltt\frac{\partial m_D}{\partial x_j}(\zeta)\frac{\partial m_D}{\partial x_k}(\zeta)+m_D(\zeta) \frac{\partial^2 m_D}{\partial x_j \partial x_k}(\zeta)\rtt  u_ju_k \\
&=&2\ltt \sum_{j=1}^{2n}\frac{\partial m_D}{\partial x_j}(\zeta) u_j\rtt^2+2m_D(\zeta) \sum_{j,k=1}^{2n}\frac{\partial^2 m_D}{\partial x_j \partial x_k}(\zeta) u_ju_k.
\eeas
Since $m_D$ is positively $1$-homogeneous, the right-hand side is $0$-homogeneous in $\zeta$,  i.e., $f(t\zeta)=f(\zeta),$ for all $\zeta\in\C^n\setminus\{0\}$ and $t>0$. 
Thus, it suffices to show that
\be\label{eq:claim str cvx}
2\ltt \sum_{j=1}^{2n}\frac{\partial m_D}{\partial x_j}(\zeta) u_j\rtt^2+2m_D(\zeta) \sum_{j,k=1}^{2n}\frac{\partial^2 m_D}{\partial x_j \partial x_k}(\zeta) u_ju_k>0,\quad\forall \zeta\in bD\,\text{and}\,\forall w\in\C^n\setminus\{0\}.
\ee
Now, as $D$ is convex, $m_D$ is a convex function on $\C^n$, which implies that 
\be\label{eq:mink cvx}
\sum_{j,k=1}^{2n}\frac{\partial^2 m_D}{\partial x_j \partial x_k}(\zeta) u_ju_k\geq 0,\quad\text{for all } \zeta\in bD\,\text{and } w\in\C^n\setminus\{0\}.
\ee
Furthermore, as $\rho_D$ is a defining function of $D$, it follows that for any $\zeta\in bD$ and $w\in \C^n\setminus\{0\}$, 
\bea\label{eq:tangent}
\text{if }
\sum_{j=1}^{2n}\frac{\partial m_D}{\partial x_j}(\zeta) u_j=0,
\quad \text{then }
\sum_{j,k=1}^{2n}\frac{\partial^2 m_D}{\partial x_j \partial x_k}(\zeta)u_ju_k>0.
\eea
Thus, \eqref{eq:claim str cvx} follows  from \eqref{eq:mink cvx} and \eqref{eq:tangent}.
\end{proof}

    \section{Some singular integral operators}\label{S:SIOs}
We now inspect some singular integral operators that are closely related to our problem. Given a bounded strongly convex domain $D\subset\Cn$, let 
        \bea
    B_D\ltt\zeta,z\rtt&=& m_D(\zeta)\left\langle 2\partial m_D(\zeta),\zeta-z\right\rangle+\ltt1-m_D^2(\zeta)\rtt,\quad \ltt\zeta,z\rtt\in\overline{D}\times\overline{D}, \label{eq:def of LS ker}\\
     \mathcal B_D(f)(z)&=&\int_{D} f(\zeta) \left(B_D\ltt\zeta,z\rtt\right)^{-n-1} dV(\zeta),\quad f\in L^p(D), z\in D.\label{eq:def of LS SIO}
        \eea  
We first establish the $L^p$-regularity of $\mathcal B_D$. 
 \begin{lemma}\label{le:bdd of LS ker}
Let $D\subset\C^n$ be a bounded strongly convex domain in $\C^n$ containing the origin, and $p\in(1,\infty)$.
   The integral operator denoted by $\left|\mathcal B_D\right|$ defined on $L^p\ltt D\rtt$ by
   \[
    \left|\mathcal B_D\right|(f)(z)=\int_{D} f(\zeta) \left|B_D\ltt\zeta,z\rtt\right|^{-n-1} dV(\zeta),\quad z\in D,
    \]
   is a bounded operator on $L^p\ltt D\rtt$.
\end{lemma}
\begin{proof} 
We claim that for $(\zeta,z)\in\overline D\times\overline D$, $\rea{B_D}(\zeta,z)$ vanishes if and only if  $\ltt\zeta,z\rtt\in bD\times bD$ and $\zeta=z$. 
When $\zeta,z\in bD$, and $\zeta=z$, clearly, $\rea B_D$ vanishes. For the converse, suppose there is a $(\zeta_0, z_0)\in\overline{D}\times\overline{D}$ such that $\rea B_D(\zeta_0, z_0)=0$. As $m_D$ is positively $1$-homogeneous, by Euler's identity, 
\be\label{eq:euler}
\rea \left\langle2\partial m_D(\zeta),\zeta\right\rangle= m_D(\zeta),\quad\forall\zeta\in \C^n.
\ee
Using this, it follows that $m_D^2(\zeta_0)-m_D(\zeta_0)\rea\left\langle2\partial m_D(\zeta_0),z_0\right\rangle=m_D^2(\zeta_0)-1$. In other words,
\be\label{eq:constraint}
m_D(\zeta_0)\rea\left\langle2\partial m_D(\zeta_0),z_0\right\rangle=1.
\ee
Let us consider the polar set of $D$, denoted as $D^\circ$, and defined by
$$D^\circ=\{z\in\C^n:H_D(z)<1\},$$
where {$H_D(\overline z)$ is the support function of $D$.}
It is known that $2\partial m_D(\zeta)\in bD^\circ$ for all $\zeta\in bD$; see {\cite[Section 2]{Li02}}, and as $\partial m_D$ is $0$-homogeneous, $2\partial m_D(\zeta)\in bD^\circ$ for all $\zeta\in\C^n\setminus\{0\}.$ Hence, by the definition of $D^\circ$, we have that $\rea\left\langle2\partial m_D(\zeta_0),z_0\right\rangle\leq 1$. Since $m_D(\zeta)<1$, whenever $\zeta\in D$, it follows from \eqref{eq:constraint} that $\ltt\zeta_0,z_0\rtt\notin{D}\times \overline D$. Thus,  $\ltt\zeta_0,z_0\rtt\in bD\times \overline{D}$. However, this means that $\rea\left\langle2\partial m_D(\zeta_0),\zeta_0-z_0\right\rangle=0$, which implies that the real tangent space to $bD$ at $\zeta_0$ intersects the point $z_0$. As $D$ is strongly convex, the real tangent space to $bD$ at $\zeta_0$ does not intersect any points in $\overline{D}\setminus\{\zeta_0\}$. Hence, $\zeta_0=z_0$, which proves our claim. 

Writing $B_D$ in terms of $\rho_D$, we have that
\[
B_D\ltt\zeta,z\rtt=\left\langle\partial \rho_D(\zeta),\zeta-z\right\rangle-\rho_D\ltt\zeta\rtt.
\]
We claim that $B_D$ satisfies the following estimate
\bea\label{eq:key est LS ker}
\left|B_D\ltt\zeta,z\rtt\right|\approx \left|\rho_D(\zeta)\right|+\left|\rho_D(z)\right|+\left|\operatorname{Im}\left\langle\partial \rho_D(\zeta),\zeta-z\right\rangle\right|+\|\zeta-z\|^2, \quad \ltt\zeta,z\rtt\in \overline{D}\times \overline{D}.
\eea
To prove this, first choose a small positive number $\delta>0$, and consider the region $G_{\delta}=\{\zeta\in\overline{D}: -\rho_D(\zeta)\leq\delta\}$.
Applying Taylor's theorem 
for each $\zeta\in G_\delta$, we get that 
\[
\rho_D(z)=\rho_D(\zeta)-2\Rea \left\langle\partial \rho_D(\zeta),\zeta-z\right\rangle+\rea Q_\zeta(\rho_D,(\zeta-z))+R_\zeta(\rho_D,(\zeta-z))+E(\zeta,z),\quad \forall z\in\overline{D},
\]
where
\beas
Q_\zeta(\rho_D,(\zeta-z))&=&\sum_{j,k=1}^{n}\frac{\partial^2 \rho_D}{\partial \zeta_j \partial \zeta_k}(\zeta) (\zeta_j-z_j)(\zeta_k-z_k),\\
R_\zeta(\rho_D,(\zeta-z))&=&\sum_{j,k=1}^{n}\frac{\partial^2 \rho_D}{\partial \zeta_j \partial \overline{\zeta_k}}(\zeta) (\zeta_j-z_j)\overline{(\zeta_k-z_k)},
\eeas
and $E(\zeta,z)$ is a continuous function on $G_\delta\times\overline D$, satisfying $\lim_{\zeta\rightarrow z}E(\zeta,z)/\|\zeta-z\|^2=0$.
Then 
\[
\widetilde{E}\ltt\zeta,z\rtt=\begin{cases}
\dfrac{E(\zeta,z)}{\|\zeta-z\|^2},&\quad \text{when}\,\zeta\neq z,\\
0,&\quad \text{when}\,\zeta=z,
\end{cases}
\]
is a uniformly continuous function on the compact set $G_\delta\times\overline D$.  
Writing $\zeta$ and $z$ in real coordinates as
$\zeta=(u_1,u_2\cdots,u_{2n})$ and $z=(x_1,x_2,\cdots,x_{2n})$, the real Hessian of $\rho_D$ at $\zeta$ is related to $Q_\zeta,R_\zeta$ in the following way:
\beas
\frac{1}{2}\sum_{j,k=1}^{2n}\frac{\partial^2 \rho_D}{\partial u_j \partial u_k}(\zeta) (u_j-x_j)(u_k-x_k)=\rea Q_\zeta(\rho_D,(\zeta-z))+R_\zeta(\rho_D,(\zeta-z)).
\eeas
Furthermore, by Lemma~\ref{le:str cvx of Minkwski}, as $\rho_D$ is a strongly convex function on $\C^n\setminus\{0\}$, we get that
\[
C_1\|\zeta-z\|^2\leq\rea Q_\zeta(\rho_D,(\zeta-z))+R_\zeta(\rho_D,(\zeta-z))\leq C_2\|\zeta-z\|^2,\quad 
(\zeta,z)\in G_\delta\times\overline D,
\]
where $C_1,C_2$ are positive constants that are independent of both $\zeta$ and $z$. 
Combining this with the Taylor expansion of $\rho_D$ above, we get that
for a fixed $\zeta\in G_\delta$,
\[
C_1\|\zeta-z\|^2+ E(\zeta,z)\leq2\Rea\ltt\left\langle\partial \rho_D(\zeta),\zeta-z\right\rangle-\rho_D(\zeta)\rtt+\rho_D(\zeta)+\rho_D(z)\leq C_2\|\zeta-z\|^2+ E(\zeta,z),\quad z\in \overline{D}.
\] 
Now, due to the uniform continuity of $\widetilde E$, there exists a $\epsilon>0$, 
such that for any $\ltt\zeta,z\rtt\in G_\delta\times\overline{D}$ satisfying $\|\zeta-z\|<\eps$,
\[
\widetilde{C_1}\leq\frac{C_1\|\zeta-z\|^2+E\ltt\zeta,z\rtt}{\|\zeta-z\|^2}\leq C_2+\widetilde{E}\ltt\zeta,z\rtt\leq \widetilde{C_2},
\]
where $\widetilde{C_1}, \widetilde{C_2}$ are constants that do not depend on $\zeta,z$.
As a result, we have that
\be\label{eq:est}
2\Rea\ltt\left\langle\partial \rho_D(\zeta),\zeta-z\right\rangle-\rho_D(\zeta)\rtt\approx-\rho_D(\zeta)-\rho_D(z)+\|\zeta-z\|^2, \quad (\zeta,z)\in G_{\eps,\frac{\delta}{2}},
\ee
where 
$G_{\eps,\frac{\delta}{2}}=\{\ltt\zeta,z\rtt\in\overline D\times\overline D; -\rho_D(\zeta)<\delta/2,\|\zeta-z\|<\eps\}$. Note that both the right-hand side and left-hand side in \eqref{eq:est} vanish if and only if $\zeta$ and $z$ are in $bD$, and $\zeta=z$. Hence, by the compactness of $\overline D\times\overline D\setminus G_{\eps,\frac{\delta}{2}}$, it follows that 
\bea\label{eq:real part est}
2\Rea\ltt\left\langle\partial \rho_D(\zeta),\zeta-z\right\rangle-\rho_D(\zeta)\rtt\approx-\rho_D(\zeta)-\rho_D(z)+\|\zeta-z\|^2,\quad\forall\ltt\zeta,z\rtt\in\overline D\times\overline D.
\eea
Thus,
\beas
\left|B_{D}\ltt\zeta,z\rtt\right|&\approx& \left|\Rea\ltt\left\langle\partial \rho_D(\zeta),\zeta-z\right\rangle-\rho_D(\zeta)\rtt\right|+\left|\operatorname{Im}\left\langle\partial \rho_D(\zeta),\zeta-z\right\rangle\right|,\\
 &\approx&
 \left|\rho_D(\zeta)\right|+\left|\rho_D(z)\right|+\left|\operatorname{Im}\left\langle\partial \rho_D(\zeta),\zeta-z\right\rangle\right|+\|\zeta-z\|^2,\quad \ltt\zeta,z\rtt\in \overline{D}\times\overline{D}.
\eeas
This proves \eqref{eq:key est LS ker}.
Now, using \eqref{eq:key est LS ker}, and following the  steps presented in \cite[Lemma 4.1]{LS12}, we conclude the $L^p$-boundedness of ${\mathcal B_D}$.
\end{proof}

Next, we consider a kernel that appears when one pushes forward the Fantappi{\`e} transform on $D^*$ under the diffeomorphism $T_{D^*}$ considered in Lemma~\ref{le:cov dom}. Let
\bea
    K_{D}\ltt\zeta,z\rtt&=& 
        1-\langle T_D(\zeta),z\rangle,\quad (\zeta,z)\in\overline D\times\overline D 
            \label{eq:def of FLD kera}\\
        &=&
        \begin{cases}        \dfrac{\left\langle \partial m_{D}(\zeta),\zeta-m_{D}^2(\zeta)z\right\rangle}{\left\langle \partial m_{D}(\zeta),\zeta\right\rangle},&\quad\ltt\zeta,z\rtt\in\ltt\overline D\setminus\{0\}\rtt\times\overline D,\\
    1,&\quad \ltt\zeta,z\rtt\in\{0\}\times\overline D,         \end{cases} \label{eq:def of FLS ker},\\
    \mathcal K_D(f)(z)&=&\int_{D} f(\zeta) \left(K_D\ltt\zeta,z\rtt\right)^{-n-1} dV(\zeta),\quad f\in L^p(D), z\in D.
    \label{eq:def of FLS SIO}
\eea

\begin{lemma}\label{le:comp of two ker}
    Let $D\subset\Cn$ be a bounded strongly convex domain containing the origin. Then, 
    \bea\label{eq:comp of}
    \left|B_D\ltt\zeta,z\rtt\right|\leq C \left|K_D\ltt\zeta,z\rtt\right|,\quad\forall\ltt\zeta,z\rtt\in\overline D\times\overline D.
    \eea
where $C>0$ are constants not depending on $\zeta,z$.
\end{lemma}
\begin{proof} Writing $K_D\ltt\zeta,z\rtt$ and $B_D(\zeta, z)$, with respect to $\rho_D$, we get that
\beas
K_D\ltt\zeta,z\rtt&=&\frac{1+\rho_D(\zeta)}{\left\langle\partial\rho_D(\zeta),\zeta\right\rangle}\ltt\left\langle\partial \rho_D(\zeta),\zeta-z\right\rangle-\frac{\rho_D(\zeta)}{1+\rho_D(\zeta)}\left\langle\partial\rho_D(\zeta),\zeta\right\rangle\rtt,\\
B_D\ltt\zeta,z\rtt&=&\left\langle\partial \rho_D(\zeta),\zeta-z\right\rangle-\rho_D\ltt\zeta\rtt.
\eeas
We write $\zeta\in \overline{D}$ as $\zeta=r\xi$, where $r\in[0,1]$ and $\xi\in bD$. As $m_D$ is positively $1$-homogeneous, $\partial m_D$ is positively $0$-homogeneous. Thus,
\be\label{eq:rho m}
\frac{1+\rho_D(\zeta)}{\left\langle\partial\rho_D(\zeta),\zeta\right\rangle}=\frac{m_D(\zeta)}{\left\langle2\partial m_D(\zeta),\zeta\right\rangle}=\frac{1}{\left\langle2\partial m_D(\xi), \xi\right\rangle}.
\ee
Now, from \eqref{eq:euler}, we have that $\rea\left\langle2\partial m_D(\xi),\xi\right\rangle=1,\,\forall\xi\in bD$. Consequently,  $\left|\left\langle2\partial m_D(\xi),\xi\right\rangle\right|\approx1$ on $bD$, which in turn implies that  $\left|\dfrac{1+\rho_D(\zeta)}{\left\langle\partial\rho_D(\zeta),\zeta\right\rangle}\right|\approx 1$ on $\overline D$. Thus, if we set
\be\label{eq:tildeK}
\widetilde K_D\ltt\zeta,z\rtt=K_D\ltt\zeta,z\rtt\frac{\left\langle\partial\rho_D(\zeta),\zeta\right\rangle}{1+\rho_D(\zeta)},\quad(\zeta,z)\in\overline D\times\overline D,
\ee
we have that
\bea\label{eq:equiv est}
\left|{K_D}\ltt\zeta,z\rtt\right|\approx \left|\widetilde K_D\ltt\zeta,z\rtt\right|,\quad\ltt\zeta,z\rtt\in \overline D\times\overline D.
\eea
Now, observe that
\beas
\left|\frac{{B_D\ltt\zeta,z\rtt}}{\widetilde K_D\ltt\zeta,z\rtt}\right|&=&\left|\frac{\left\langle\partial \rho_D(\zeta),\zeta-z\right\rangle-\rho_D(\zeta)}{\left\langle\partial \rho_D(\zeta),\zeta-z\right\rangle-\dfrac{\rho_D(\zeta)}{1+\rho_D(\zeta)}\left\langle\partial\rho_D(\zeta),\zeta\right\rangle}\right|\\
&=&\left|\frac{\left\langle\partial \rho_D(\zeta),\zeta-z\right\rangle-\dfrac{\rho_D(\zeta)}{1+\rho_D(\zeta)}\left\langle\partial\rho_D(\zeta),\zeta\right\rangle+\rho_D(\zeta)\ltt\dfrac{\left\langle\partial\rho_D(\zeta),\zeta\right\rangle}{1+\rho_D(\zeta)}-1\rtt}{\left\langle\partial \rho_D(\zeta),\zeta-z\right\rangle-\dfrac{\rho_D(\zeta)}{1+\rho_D(\zeta)}\left\langle\partial\rho_D(\zeta),\zeta\right\rangle}\right|\\
&\leq&1+\frac{\left|\rho_D(\zeta)\right|\left|\dfrac{\left\langle\partial\rho_D(\zeta),\zeta\right\rangle}{1+\rho_D(\zeta)}-1\right|}{\left|\left\langle\partial \rho_D(\zeta),\zeta-z\right\rangle-\dfrac{\rho_D(\zeta)}{1+\rho_D(\zeta)}\left\langle\partial\rho_D(\zeta),\zeta\right\rangle\right|}\\
&\lesssim& 1+\frac{\left|\rho_D(\zeta)\right|}{\left|\left\langle\partial \rho_D(\zeta),\zeta-z\right\rangle-\dfrac{\rho_D(\zeta)}{1+\rho_D(\zeta)}\left\langle\partial\rho_D(\zeta),\zeta\right\rangle\right|}\\
&\leq&1+\frac{\left|\rho_D(\zeta)\right|}{\left|\Rea\ltt\left\langle\partial \rho_D(\zeta),\zeta-z\right\rangle-\dfrac{\rho_D(\zeta)}{1+\rho_D(\zeta)}\left\langle\partial\rho_D(\zeta),\zeta\right\rangle\rtt\right|}\\
(\text{by \eqref{eq:rho m}})
&=&
1+\frac{\left|\rho_D(\zeta)\right|}{\left|\Rea\ltt\left\langle\partial \rho_D(\zeta),\zeta-z\right\rangle-\rho_D(\zeta)\rtt\right|}\\
(\text{by \eqref{eq:real part est}})
&\leq& 1+\frac{\left|\rho_D(\zeta)\right|}{C_1\ltt\left|\rho_D(\zeta)\right|+\left|\rho_D(z)\right|+\|\zeta-z\|^2\rtt}\\
&\leq& 1+\frac{1}{C_1}.
\eeas
Finally, combining this with \eqref{eq:equiv est}, we conclude the lemma.
\end{proof}



   We now discuss a reproducing kernel on $p$-Bergman spaces of strongly convex domains. Define a $(1,0)$-form as follows. For $\ltt\zeta,z\rtt\in\overline{D}\times D$, consider
    \be\label{eq:repr ker}
G\ltt\zeta,z\rtt=\dfrac{\partial\rho_D(\zeta)}{\widetilde{K}_D\ltt\zeta,z\rtt},
    \ee
    where $\widetilde K_D$ is as in \eqref{eq:tildeK}. Since $\rho_D$ is $\cont^2$ smooth on $\C^n\setminus\{0\}$, it follows that for each $z\in D$, $G$ is a $C^1$-smooth form on $\overline  D\setminus\{0\}$. We show that the $(n,n)$-form $\ltt\overline{\partial}_{\zeta}G\rtt^{n}(\zeta,z)$ reproduces functions in $A^p( D)$. For this, we exploit the reproducing property of the Cauchy--Leray operator for functions in the dense subclass $\hol\ltt D\rtt\cap C^1\ltt\overline D\rtt$.
    
    \begin{lemma}\label{le:repr}
   Let $D\subset\Cn$ be a bounded strongly convex domain. Given $p\in(1,\infty)$, let $f\in A^p\ltt D\rtt$. Then, for each $z\in D$, 
    \be\label{eq:repr prop}
    f(z)=\frac{1}{(2\pi i)^n}\int_D f(\zeta)\ltt\overline{\partial}G\rtt^n\ltt\zeta,z\rtt.
    \ee
    \end{lemma}
    \begin{proof}
      Consider the modified $(1,0)$-form 
        \beas
L\ltt\zeta,z\rtt=\dfrac{\partial\rho_D(\zeta)}{\widetilde{K}_D\ltt\zeta,z\rtt+\frac{\rho_D(\zeta)}{1+\rho_D(\zeta)}\left\langle\partial\rho_D(\zeta),\zeta\right\rangle}=\dfrac{\partial\rho_D(\zeta)}{\left\langle\partial\rho_D(\zeta),\zeta-z\right\rangle},\quad \ltt\zeta,z\rtt\in\overline{D}\times D.
        \eeas
This form is a generating form for a convex domain with $\cont^2$-smooth boundary; see\cite[\S 3]{Ra98}. Thus the $(n,n-1)$-form, $L\wedge\ltt\overline\partial L\rtt^{n-1}$ is a Cauchy--Fanatppi{\`e} form on $bD$, and reproduces functions in $\hol\ltt D\rtt\cap\cont\ltt\overline D\rtt$ from its boundary values. As a consequence, we get that for any $f\in\hol\ltt D\rtt\cap\cont^1\ltt\overline{D}\rtt$, 
\begin{align}\label{eq:repr larey}\nonumber
f(z)&=\frac{1}{(2\pi i)^n}\int_{bD} f(\zeta)j^*\ltt L\wedge\ltt\overline\partial L\rtt^{n-1}\rtt(\zeta,z)\\
&=\frac{1}{(2\pi i)^n}\int_{bD} \dfrac{f(\zeta)}{\left\langle\partial\rho_D(\zeta),\zeta-z\right\rangle^n} j^*\ltt\partial \rho_D\wedge\ltt\overline{\partial}\partial\rho_D\rtt^{n-1}\rtt(\zeta), \quad z\in D,
\end{align}
where $j:bD\rightarrow\C^n$ is the inclusion map. It is easy to verify that
\[
j^*\ltt L\wedge\ltt\overline\partial L\rtt^{n-1}\rtt(\zeta,z)=j^*\ltt G\wedge\ltt\overline\partial G\rtt^{n-1}\rtt(\zeta,z),\quad \ltt\zeta,z\rtt\in bD\times D.
\]
Hence by \eqref{eq:repr larey},
\be\label{eq:repr bdry our}
f(z)=\frac{1}{(2\pi i)^n}\int_{bD} f(\zeta)j^*\ltt G\wedge\ltt\overline\partial G\rtt^{n-1}\rtt(\zeta,z),\quad z\in D.
\ee
 Now, to conclude \eqref{eq:repr prop} for $f$, we wish to apply Stokes' theorem to the term on the right-hand side of the above equality. However, since the form $\ltt \overline \partial G\rtt^{n}$ is not $\cont^1$-smooth at the origin, we cannot apply Stokes' theorem directly. To get around this issue, let us first fix a $z\in D$. Consider $\lambda\in(0,1)$, such that $z\in  D\setminus \overline{\lambda D}$, where $\lambda D=\{\lambda\zeta:\zeta\in D\}$. Applying Stokes' theorem on the domain $ D\setminus \overline{\lambda D}$, we obtain that
\be\label{eq:int decomp}
\int_{ D\setminus \overline{\lambda D}}f(\zeta) \ltt \overline \partial G\rtt^n\ltt\zeta,z\rtt=\int_{b D} f(\zeta)j^*\ltt G\wedge\ltt\overline\partial G\rtt^{n-1}\rtt(\zeta,z)\minus\int_{\lambda b D}  f(\zeta)j^*\ltt G\wedge\ltt\overline\partial G\rtt^{n-1}\rtt(\zeta,z).
\ee
By a straightforward  computation, we can verify that 
 \bea\label{eq:comp vol}
 \ltt\overline\partial G\rtt^n(\zeta,z)
&=& 
-\frac{\ltt\ltt\overline\partial\partial\rho_D\rtt^{n-1}\wedge\overline{\partial}\widetilde K_ D\wedge\partial\rho_D+\widetilde K_ D\ltt\overline{\partial}\partial\rho_D\rtt^n\rtt}{\widetilde K_ D\ltt\zeta,z\rtt^{n+1}}\nonumber\\
&=&\dfrac{\mathfrak h_ D\ltt\zeta\rtt}{\widetilde K_ D\ltt\zeta,z\rtt^{n+1}}dV(\zeta),\quad \ltt\zeta,z\rtt\in\ltt D\setminus\{0\}\rtt\times D,
 \eea
 where
\bes
\renewcommand\arraystretch{2.5}
\mathfrak h_ D(\zeta)= c_n\det
\begin{pmatrix}
\dfrac{\rho_D(\zeta){\left\langle\partial\rho_D(\zeta),\zeta\right\rangle}}{1+\rho_D(\zeta)} & \dfrac{\partial\rho_D}{\partial \zeta_j}\\
    \dfrac{\partial}{\partial \bar \zeta_j}\ltt\dfrac{\rho_D(\zeta)\left\langle\partial\rho_D(\zeta),\zeta\right\rangle}{1+\rho_D(\zeta)}\rtt & 
    \dfrac{\partial^2\rho_D}{\partial \zeta_j\partial \bar \zeta_k}
\end{pmatrix}_{1\leq j,k\leq n}, \quad \zeta\in  D.
\ees
Similarly, using \cite[Lemma~3.4.]{Ra98}, we have that
\begin{align}\label{eq:leray surf area}
j^*\ltt G\wedge\ltt\overline\partial G\rtt^{n-1}\rtt(\zeta,z)
=\frac{j^*\ltt\partial \rho_D\wedge\ltt\overline\partial\partial\rho_D\rtt^{n-1}\rtt}{\widetilde{K}_ D\ltt\zeta,z\rtt^{n}}\nonumber=\dfrac{\mathfrak h_{bD}\ltt\zeta\rtt}{\widetilde{K}_ D\ltt\zeta,z\rtt^{n}}\sigma_ D,
\end{align}
where $\sigma_ D$ is the surface area measure on $b D$, and
\[
\mathfrak{h}_{b D}\ltt\zeta\rtt=\frac{(-1)^n}{\pi^n}\det\begin{pmatrix}
    0 & \dfrac{\partial\rho_D}{\partial \zeta_j}\\
    \dfrac{\partial\rho_D}{\partial \bar \zeta_j} & 
    \dfrac{\partial^2\rho_D}{\partial \zeta_j\partial \bar \zeta_k}
\end{pmatrix}_{1\leq j,k\leq n}, \quad \zeta\in bD.
\]
Since $D$ is strongly convex, $\mathfrak{h}_{b D}$ is bounded above and below by positive constants.
Furthermore, since $m_D$ is $1$-homogeneous, the second order partial derivatives of $\rho_D$ are $0$-homogeneous, which implies that $\sup_{\zeta\in\overline D} \frac{\partial^2\rho_D}{\partial \zeta_j\partial \bar \zeta_k}(\zeta)=\sup_{\zeta\in b D}\frac{\partial^2\rho_D}{\partial \zeta_j\partial \bar \zeta_k}(\zeta)$. Consequently, as $\rho_D$ is $\cont^2$-smooth on $bD$, it follows that $\frac{\partial^2\rho_D}{\partial \zeta_j\partial \bar \zeta_k}$ is bounded above on $\overline  D$. Hence 
\be\label{eq:density bdd abv}
\left|\mathfrak h_ D(\zeta)\right|\leq C,\quad \zeta\in\overline  D,
\ee
where $C$ is a positive constant.
As mentioned earlier, since $z\in D$ is fixed $\widetilde K_ D\ltt\zeta,z\rtt\neq 0,\, \forall\zeta\in\overline D$, and hence by continuity $\left|\frac{1}{\widetilde K_ D\ltt\zeta,z\rtt}\right|\leq C_z$, where $C_z$ is a constant depending on $z$. Thus combining this with \eqref{eq:comp vol}, \eqref{eq:density bdd abv}, and the fact $f\in \cont\ltt\overline D\rtt$, we get that
\beas
\int_{ D}\left|f(\zeta)\right| \left|\ltt \overline \partial G\rtt^n\right|\ltt\zeta,z\rtt\lesssim \text{vol}(D).
\eeas
Thus, applying the dominated convergence theorem, we get that, as $\lambda\rightarrow 0$, 
\[
\int_{ D\setminus \overline{\lambda D}}f(\zeta) \ltt \overline \partial G\rtt^n\ltt\zeta,z\rtt\rightarrow \int_{ D}f(\zeta) \ltt \overline \partial G\rtt^n\ltt\zeta,z\rtt.
\]
Applying the change of variable $\zeta=\lambda\zeta'$, where $\zeta'\in b D$, we get that
\beas
\left|\int_{\lambda b D}  f(\zeta)j^*\ltt G\wedge\ltt\overline\partial G\rtt^{n-1}\rtt(\zeta,z)\right|&\lesssim&\int _{\lambda b D}j^*\ltt\partial \rho_D\wedge\ltt\overline\partial\partial\rho_D\rtt^{n-1}\rtt(\zeta),\\
&\lesssim& \lambda^{2n}\int_{b D} j^*\ltt\partial \rho_D\wedge\ltt\overline\partial\partial\rho_D\rtt^{n-1}\rtt(\zeta'),\\
&\lesssim& \lambda^{2n} \sigma\ltt b D\rtt.
\eeas
The right-hand side tends to $0$, as $\lambda$ tends to $0$. Hence, $\lim_{\lambda\rightarrow 0}\int_{\lambda b D}  f(\zeta)j^*\ltt G\wedge\ltt\overline\partial G\rtt^{n-1}\rtt(\zeta,z)\rightarrow 0$. Combining these observations with \eqref{eq:int decomp} and \eqref{eq:repr bdry our}, we get that
 \[
 f(z)=\frac{1}{(2\pi i)^n}\int_ D f(\zeta)\ltt\overline{\partial}G\rtt^n\ltt\zeta,z\rtt,\quad z\in D.
 \]

Since $ D$ is a convex domain, the class $\hol\ltt D\rtt\cap\cont^1\ltt\overline D\rtt$ is dense in $A^p\ltt D\rtt$. Hence to prove \eqref{eq:repr prop} for $A^p( D)$ functions, where $p\in\ltt1,\infty\rtt$, it is enough to show that for every $p\in\ltt1,\infty\rtt$, the integral operator
\[
L^p( D)\ni f\mapsto\frac{1}{(2\pi i)^n}\int_ D f\ltt\zeta\rtt \ltt \overline \partial G\rtt^n\ltt\zeta,z\rtt,
\]
is a bounded operator on $L^p\ltt D\rtt$. To see this, for $f\in L^p\ltt D\rtt$,
\beas
\left\|\frac{1}{(2\pi i)^n}\int_ D f\ltt\zeta\rtt \ltt \overline \partial G\rtt^n\ltt\zeta,\cdot\rtt\right\|_{L^p(D)}^p&=&\left\|\frac{1}{(2\pi i)^n}\int_ D f\ltt\zeta\rtt \frac{\mathfrak h_D(\zeta)}{\widetilde K_D(\zeta,\cdot)^{n+1}}dV(\zeta)\right\|_{L^p(D)}^p\\
(\text{by}\,\eqref{eq:comp vol},\eqref{eq:density bdd abv})&\lesssim&\left\|\int_ D\left|f\ltt\zeta\rtt\right|\left|K_ D\ltt\zeta,\cdot\rtt\right|^{-n-1} dV(\zeta)\right\|_{L^p(D)}^p\\
(\text{by Lemma~\ref{le:comp of two ker}})&\lesssim&\left\|\int_ D\left|f\ltt\zeta\rtt\right|\left|B_ D\ltt\zeta,\cdot\rtt\right|^{-n-1} dV(\zeta)\right\|_{L^p(D)}^p\\
(\text{by Lemma~\ref{le:bdd of LS ker}})&\lesssim& \|f\|_{L^p(D)}^p.
\eeas
This proves our claim, and hence completes the proof of the lemma.
\end{proof}

Next, we note the following duality result on strongly convex domains, which is implicitly present in the literature. In fact, the result holds for all $\cont^2$-smooth strongly pseudoconvex domains. 
   


\begin{lemma}\label{le:dual of Ap}
Let $D\subset\C^n$ be a bounded $\cont^2$-smooth strongly $\C$-convex domain. Given $p\in(1,\infty)$, let $q\in(1,\infty)$ be the conjugate exponent of $p$, i.e., $p^{-1}+q^{-1}=1$. Then, the map
$\Phi_p:A^q\ltt D\rtt\rightarrow A^p( D)'$, defined as
\[
\Phi_p(f)(h)=\int_{D} h(\zeta)\overline{f(\zeta)} dV(\zeta), \quad f\in A^q\ltt D\rtt,\, h\in A^p\ltt D\rtt,
\]
is an isomorphism.
\end{lemma}

\begin{proof}
{Fix $p\in(1,\infty)$.
According to Theorem~2.15 and Remark~2.6 in \cite{CEM19},  the surjectivity of $\Phi_p$ is guaranteed if the following two conditions are satisfied:
\begin{itemize}
    \item [(i)]For $p \in(1, \infty)$, the positive Bergman operator $|P|$, defined by
    \[
    |P|(f)(z)=\int_D f(\zeta) | \widetilde{B}_D(\zeta,z)| dV(\zeta),\quad z\in D, \,f\in L^p(D),
    \]
    where $\widetilde{B}_D$ is the Bergman kernel of the domain $D$, is a bounded operator on $L^p(D)$.
     \item[(ii)] The Bergman projection acts as the identity operator on $A^p(D)$.
\end{itemize}
  Both conditions are known to hold for bounded $\cont^2$-smooth strongly $\mathbb{C}$-convex domains. Specifically, condition (i) was established in \cite{LS12}, while by \cite[Lemma 2.8]{CEM19} and condition (i), condition (ii) is equivalent to $A^2(\Omega)\cap A^p(\Omega)$ being dense in $A^p(\Omega)$, which follows from \cite[Proposition 7.1]{LS12}.}
  Following the same argument but switching the roles of $p$ and $q$, we obtain that $\Phi_q$ is also surjective. Hence, from {\cite[Corollary 2.14]{CEM19}}, it follows that $\Phi_p$ is injective, which completes the proof.
\end{proof} 

Next, we introduce the Borel transform on $A^2(\C^n,\omega_D)$, which allows us to relate $\F_{n+1}$ with $\mathcal L$, to establish Theorem~\ref{th:PW bergman}.

Let $D\subset\C^n$ be a bounded convex domain. Let $\hol_{\operatorname{\operatorname{exp}}}(D)$ denote the space of entire functions $F\in\hol(\C^n)$ such that, for every $\eps>0$, there exist a $C_\eps>0$ so that
\be\label{eq:Borel dfn cond}
|F(z)|\leq C_\epsilon e^{(1+\eps)H_D(z)},\quad z\in\C^n.
\ee The Borel transform on $\hol_{\operatorname{\operatorname{exp}}}(D)$ is defined as
\be\label{eq:def borel}
\mathfrak B_n\ltt F\rtt(z)=\int_0^\infty F(tz)e^{-t}t^n dt,\quad F\in \hol_{\operatorname{\operatorname{exp}}}(D).
\ee
It follows from the definition that $\mathfrak B_n(F)\in\hol(D^\circ)$, where $D^{\circ} $ is the polar set of $D$. Furthermore, it is known that $\mathfrak B_n(F)$ has a holomorphic extension to $D^*$; see \cite{MA67}, or \cite[page~139]{APS}. Next, we show that for planar strongly convex domain, $\mathcal B_1$ is an isomorphism between  $A^2(\C,\omega_D)$ and $A^2(D^*)$, where recall that $\omega_D=e^{-2H_\Om(z)}\|z\|^{\frac{3}{2}}\ltt dd^cH_D\rtt$. This result essentially follows from \cite[Preliminary theorem]{IY04}. However, the weighted Bergman space appearing in \cite{IY04} is different from $\omega_D$. 
  Thus, in order to establish the isomorphism between $A^2(\C,\omega_D)$, and $A^2(D^*)$, we first show that the weighted Bergman space appearing in \cite{IY04}, and $A^2(\C,\omega_D)$ are normed space isomorphic.
  \begin{lemma}\label{le:comp berg sp}
    Let $D\subset\C^n$ be a bounded strongly convex domain with $0\in D$. Then the identity map is a normed space isomorphism between $A^2(\C^n,\omega_D)$ and $A^2(\C^n,\mu_D)$, where 
    \beas
    A^2(\C^n,\mu_D)=\left\{F\in\hol(\C^n):\|F\|_{\mu_D}^2=\int_{\C^n}{|F(z)|^2}{\left\|e^{\langle\cdot, z\rangle}\right\|^{-2}_{L^2(D)}}\ltt dd^cH_D\rtt^n(z)<\infty\right\},
    \eeas
    equipped with the norm $\|.\|_{\mu_D}$.
\end{lemma}
\begin{proof}
    Following the same idea as in the proof of \cite[Proposition 7.2]{AC24}, it is sufficient to prove that for any $z\in\C^n$, satisfying $\|z\|\geq1$,
       \be\label{eq:comp weight}
       \left\|e^{\langle\cdot, z\rangle}\right\|^2_{L^2(D)}\approx {e^{2H_D(z)}}\|z\|^{-n-\frac{1}{2}}.
       \ee 
       To show this, by applying \eqref{eq:polar coordinate} to $D$, we get that
       \bea
       {e^{-2H_D(z)}}\left\|e^{\langle\cdot, z\rangle}\right\|^2_{L^2(D)}&\approx&  { \int_0^1\int_{bD} e^{2r\rea\langle\xi,z\rangle-2H_D(z)}} r^{2n-1} d\sigma_D(\xi) dr\nonumber\\
       \label{eq:berg eq key est} &=& \int_0^1 e^{2(r-1)H_D(z)} r^{2n-1}\int_{bD} e^{2r\ltt\rea\langle\xi,z\rangle-H_D(z)\rtt} d\sigma_D(\xi) dr. 
       \eea
       Also, applying \cite[(7.7)-(7.11)]{AC24}, we obtain that for any $z\in\C^n$,  satisfying $\|z\|\geq1$, and $r\in (0,1)$,
       \beas
      \int_{b\mathbb B^n} e^{-2C_1\|z\|\ltt1-\rea\langle\eta,z\rangle\rtt} d\sigma_{\mathbb B^n}(\eta)\lesssim \int_{bD} e^{2r\ltt\rea\langle\xi,z\rangle-H_D(z)\rtt} d\sigma_D(\xi)\lesssim \int_{b\mathbb B^n} e^{-2rC_2\|z\|\ltt1-\rea\langle\eta,z\rangle\rtt} d\sigma_{\mathbb B^n}(\eta),
       \eeas
       where $C_1,C_2$ are positive constants not depending on either $r$ or $z$. Furthermore, following the same computations as in the proof of \cite[Lemma~7.1]{AC24} to obtain \cite[(7.4),(7.5)]{AC24}, we get that for any $z\in\C^n$,  satisfying $\|z\|\geq1$, and $r\in (0,1)$, 
       \beas
       \|z\|^{n-\frac{1}{2}}\int_{b\mathbb B^n} e^{-2C_1\|z\|\ltt1-\rea\langle\eta,z\rangle\rtt} d\sigma_{\mathbb B^n},(\eta)&\gtrsim& 1\\
        r^{n-\frac{1}{2}}\|z\|^{n-\frac{1}{2}}\int_{b\mathbb B^n} e^{-2rC_2\|z\|\ltt1-\rea\langle\eta,z\rangle\rtt} d\sigma_{\mathbb B^n}(\eta)&\lesssim& 1.
        \eeas
     From \eqref{eq:berg eq key est}, it now follows that for any $z\in\C^n$,  satisfying $\|z\|\geq1$, 
        \bea\label{eq:berg eq prel est}
        \|z\|\int_0^1 e^{-{(1-r)H_D(z)}}r^{2n-1} dr\lesssim\|z\|^{n+\frac{1}{2}}{e^{-2H_D(z)}}\|e^{\cdot z}\|^2_{L^2(D)}\lesssim \|z\|\int_0^1 e^{-{(1-r)H_D(z)}} r^{n-\frac{1}{2}} dr.
        \eea
        {
        Moreover, as there exists $r_1,r_2>0$, such that $\mathbb B^n(r_1)\subset D\subset \mathbb B^n(r_2)$, for any $z\in \C^n\setminus\{0\}$, 
        \bea\label{eq:support radial comp}
        r_1 \|z\|\leq H_D(z)\leq r_2\|z\|.
        \eea
Combining \eqref{eq:support radial comp} with the change of variable $r'=(1-r)H_D(z)$ yields that for any $z\in\C^n$,  satisfying $\|z\|\geq1$,
        \beas
     \|z\|  \int_0^1 e^{-{(1-r)H_D(z)}}r^{2n-1} dr&\approx&
     H_D(z)  \int_0^1 e^{-{(1-r)H_D(z)}}r^{2n-1} dr\\
     &=& \int_0^{H_D(z)} e^{-r'}\ltt1-\frac{r'}{H_D(z)}\rtt^{2n-1} dr'\\
        &\leq& \int_0^{H_D(z)} e^{-r'} dr'\\
        &=&1-e^{-H_D(z)}\lesssim 1,
        \eeas
       Also, by \eqref{eq:support radial comp}, $\|z\|\geq 1$, implies $H_D(z)\geq r_1$. Thus, for $z\in\C^n$, satisfying $\|z\|\geq 1$,
\beas
\int_0^{H_D(z)} e^{-r'}\ltt1-\frac{r'}{H_D(z)}\rtt^{2n-1} dr'&\geq& \ltt\frac{1}{2}\rtt^{2n-1}\int_0^{\frac{r_1}{2}} e^{-r'} dr'\approx 1. 
\eeas
Hence, we get that for any $z\in\C^n$,  satisfying $\|z\|\geq1$,
        \beas
     \|z\|  \int_0^1 e^{-{(1-r)H_D(z)}}r^{2n-1} dr&\approx& 1.
     \eeas
     A similar computation yields, for any $z\in\C^n$,  satisfying $\|z\|\geq1$,
\beas
          \|z\| \int_0^1 e^{-{(1-r)H_D(z)}}r^{n-\frac{1}{2}} dr &\approx& 1.
        \eeas
        Finally, combining these with \eqref{eq:berg eq prel est}, we obtain \eqref{eq:comp weight}.}
\end{proof}
\begin{lemma}\label{le:borel one d}
    Let $D\subset\C$ be a bounded strongly convex domain with $0\in D$. Then $\mathfrak B_1$ is a normed space isomorphism between $A^2(\C,\omega_D)$ and $A^2(D^*)$. 
    \end{lemma}
    \begin{proof}
    Let us consider the following space:
       \[
       B_2^1\ltt\C\setminus D\rtt=\left\{f\in\hol( \widehat{\C}\setminus D): f(\infty)=0,\,\int_{\C\setminus D}|f'(\lambda)|^2 dV(\lambda)<\infty\right\},
       \]
       equipped with the norm $\|f\|_{B_2^1\ltt\C\setminus D\rtt}^2=\ltt\int_{\C\setminus D}|f'(\lambda)|^2 dV(\lambda)\rtt$.
       Since $D^*$ is biholomorphic to $\widehat{\C}\setminus D$ via the map $\tau:\lambda\rightarrow\frac{1}{\lambda}$, the following map
       \beas
       \mathfrak T: &B_2^1\ltt\C\setminus D\rtt\rightarrow A^2(D^*),\\
       & F\rightarrow (F\circ \tau)',
       \eeas
       is a normed space isomorphism between $B_2^1\ltt\C\setminus D\rtt$, and $A^2(D^*)$.
    From \cite[Preliminary theorem]{IY04}, it follows that $\widetilde {\mathfrak B}:A^2(\C,\mu_D)\rightarrow B_2^1(\C\setminus D)$, defined as
    \[
    \widetilde{\mathfrak B}(F)(z)=\int_0^\infty F(t) e^{-tz} dt,
    \]
is a normed space isomorphism between $A^2(\C,\mu_D)$ and $B_2^1(\C\setminus D)$. Note that, in \cite{IY04}, the inverse of $\widetilde{\mathfrak{B}}$ has been considered as the Borel transform. For $z\in D^*$, applying a change of variable $t'=\frac{t}{z}$ we get that  
\beas
 \ltt\mathfrak T\circ\widetilde{\mathfrak B}(F)\rtt(z) &=& \frac{1}{z^2}\int_0^\infty F(t)e^{-t/z}t dt\\
&=&\int_0^\infty F(t'z) e^{-t'} t'dt'\\
&=& \mathfrak B_{1}(F)(z).
\eeas
Combining all these with Lemma~\ref{le:comp berg sp}, it follows that, $\mathfrak B_1$ is a normed space isomorphism between $A^2(\C,\omega_D)$ and $A^2(D^*)$. 
\end{proof}

\begin{section}{Proof of Theorem~\ref{th:main1}}
Fix $p\in(1,\infty)$. As discussed in the introduction, the proof splits into three parts: the $L^p$-boundedness, the injectivity, and the surjectivity of $\mathcal F_{n+1}$. 

\subsection{$\mathbf{L^p}$-boundedness of the Fantappi{\`e} transform.}
By Lemma~\ref{le:cov dom}, applied to $D=\Om^*$, 
{
\beas
\mathcal F_{n+1}\ltt f\rtt(z)&=&\frac{n!}{\pi^n}\int_{\Om} \frac{\overline{ f(\zeta)}}{\ltt 1-\langle \zeta,z\rangle\rtt^{n+1}}dV(\zeta)\\
&=& \frac{n!}{\pi^n}\int_{\Om^*} \dfrac{\overline{\ltt f\circ T_{\Om^*}\rtt}(\zeta)}{\ltt 1-\langle T_{\Om^*}(\zeta),z\rangle\rtt^{n+1}} h_{\Om^*}(\zeta) dV(\zeta) \\
&=&
\frac{n!} {\pi^n}\int_{\Om^*}\dfrac{\overline{\ltt f\circ T_{\Om^*}\rtt}(\zeta)\left\langle \partial m_{\Om^*}(\zeta),\zeta\right\rangle^{n+1} h_{\Om^*}(\zeta)}{\ltt\left\langle \partial m_{\Om^*}(\zeta),\zeta\right\rangle-m_{\Om^*}^2(\zeta)\left\langle \partial m_{\Om^*}(\zeta),z\right\rangle\rtt^{n+1}} dV(\zeta)\\
&=&\frac{n!} {\pi^n}\int_{\Om^*}\dfrac{\overline{\ltt f\circ T_{\Om^*}\rtt}(\zeta)\left\langle \partial m_{\Om^*}(\zeta),\zeta\right\rangle^{n+1} h_{\Om^*}(\zeta)}{\left\langle \partial m_{\Om^*}(\zeta),\zeta-m_{\Om^*}^2(\zeta)z\right\rangle^{n+1}} dV(\zeta)\\
&=& \frac{n!}{\pi^n}\int_{\Om^*} \overline{\ltt f\circ T_{\Om^*}\rtt}(\zeta) h_{\Om^*}(\zeta) K_{\Om^*}\ltt \zeta,z\rtt^{-n-1} dV(\zeta),
\quad f\in A^p\ltt\Om\rtt, z\in\Om^*,
\eeas
}
where $K_{\Om^*}$ is as given by \eqref{eq:def of FLS ker}.

Now for a bounded strongly convex domain $D\subset\Cn$ containing the origin, let 
\bes 
|\mathcal K_D|(f)(z)=\int_{\Om} f(\zeta) |K_D\ltt\zeta,z\rtt|^{-n-1} dV(\zeta),\quad z\in D.
\ees
By Lemma~\ref{le:comp of two ker}, we have that for any $g\in L^p\ltt D\rtt$,
\be\label{eq:int comp}
\left|\mathcal{K}_{D}\right|\ltt\left|g\right|\rtt(z)\leq \tilde{C}\left|\mathcal{B}_{D}\right|\ltt\left|g\right|\rtt(z), \quad \forall z\in D,
\ee
where $\tilde{C}$ is a positive constant, which is independent of $z$ and $g$. 

Applying \eqref{eq:int comp} to $D=\Om^*$, we get for $g\in L^p\ltt\Om\rtt$,
   \beas
   \|\mathcal F_{n+1}(g)\|_{L^p(\Om^*)}&=&\left\|\mathcal K_{\Om^*}\ltt\overline{\ltt g\circ T_{\Om^*}\rtt} h_{\Om^*}\rtt\right\|_{L^p(\Om^*)}\\
   &\leq& \left\|\left|\mathcal K_{\Om^*}\right|\ltt\left|\overline{\ltt g\circ T_{\Om^*}\rtt}\right|\left|h_{\Om^*}\right|\rtt\right\|_{L^p(\Om^*)}\\
   &\lesssim& \left\|\left|\mathcal B_{\Om^*}\right|\ltt\left|\overline{\ltt g\circ T_{\Om^*}\rtt}\right|\left|h_{\Om^*}\right|\rtt\right\|_{L^p(\Om^*)}\\
   &\lesssim&\left\| \overline{\ltt g\circ T_{\Om^*}\rtt}h_{\Om^*}\right\|_{L^p\ltt\Om^*\rtt}\\
   &\lesssim& \left\| g\right\|_{L^p\ltt\Om\rtt}
   \eeas
   where the last inequality follows due to the boundedness of $h_{\Om^*}$. This shows that $\mathcal F_{n+1}$ is a bounded operator from $A^p\ltt\Om\rtt$ to $L^p\ltt\Om^*\rtt$, $\forall p\in(1,\infty)$.
\end{section}
\subsection{Injectivity of the Fantappi{\`e} transform} Recall that $\mathcal F_{n+1}$ on $A^p(\Om)$ is, in fact, the composition
\bes
A^p(\Om)\overset{\iota} {\hookrightarrow}L^p(\Om)\xrightarrow\theta\hol'(\overline{\Om})\xrightarrow{\mathcal F_{n+1}} \hol(\Om^*),
\ees
where $\iota$ is the inclusion map, and $\theta$ is as in \eqref{eq: theta}. By the Martineau--Aizenberg duality theorem, the right-most map is already known to be injective. Thus, we must show that, for $g\in A^p(\Om)$, if  $(\theta\circ\iota)(g)=\mu_g$ given by
\bes
\mu_g:f\mapsto \int_{\Om}\overline {g(\zeta)} f(\zeta)dV(\zeta),\,\quad f\in \hol\ltt\overline\Om\rtt
\ees
is the zero map, then $g\equiv 0$. Let $g\in A^p\ltt\Om\rtt$, such that $\mu_{g}=0$. 
For any $f$ in the dense subspace $\hol\ltt\overline\Om\rtt \subset A^q\ltt\Om\rtt$, $\mu_g(f)=\overline{\Phi_g(f)}$, where $\Phi_g$ is as in Lemma~\ref{le:dual of Ap}. Thus, $\Phi_g\equiv 0$, which implies that $g\equiv 0$. This completes the proof of injectivity.
   
\subsection{Surjectivity of the Fantappi{\`e} transform}
Let $g\in A^p\ltt\Om^*\rtt$, where $p\in\ltt1,\infty\rtt$. We define a linear functional $F_g$ on $A^q\ltt\Om\rtt$, where $p^{-1}+q^{-1}=1$, as follows
\[
F_g(\psi)=\frac{1}{(2i)^nn!}\int_{\Om} \psi\ltt\zeta\rtt \ltt g\circ T_\Om\rtt\ltt\zeta\rtt\ltt\mathfrak H\circ T_\Om\rtt\ltt\zeta\rtt dV(\zeta),\quad \forall \psi\in A^q\ltt\Om\rtt,
\]
where $\mathfrak H:\Om^*\rightarrow\C$, defined as 
\be\label{eq:mathfrak H}
\mathfrak H\ltt\eta\rtt= {\mathfrak h_{\Om^*}(\eta)}{h_{\Om}\ltt T_{\Om^*}\ltt\eta\rtt\rtt}\ltt\frac{1+\rho_{\Om^*}(\eta)}{\left\langle\partial \rho_{\Om^*}(\eta),\eta\right\rangle}\rtt^{n+1},
\ee
where $\mathfrak h_{\Om^*}$, $ h_{\Om}$ are as in \eqref{eq:comp vol} and  \eqref{eq:cov}, respectively, and $T_\Om, T_{\Om^*}$ are as in \eqref{eq:cov map}. 
Since $h_{\Om}(T_{\Om^*}(\eta))$, $\left|\mathfrak h_{\Om^*}(\eta)\right|$, and $\left|\frac{1+\rho_{\Om^*}(\eta)}{\left\langle\partial \rho_{\Om^*}(\eta),\eta\right\rangle}\right|$ are bounded above by positive constants on $\Om^*$,  $\left|\mathfrak H\right|$ is bounded above on $\Om^*$. Thus, it follows that
\beas
\left|F_g\ltt\psi\rtt\right|&\lesssim&\ltt \int_\Om\left| g\circ T_\Om\right|^p\ltt\zeta\rtt dV\ltt\zeta\rtt\rtt^{\frac{1}{p}}\|\psi\|_{L^q\ltt\Om\rtt},\\
&\lesssim& \|g\|_{L^p\ltt\Om^*\rtt}\|\psi\|_{L^q\ltt\Om\rtt}.
\eeas
This shows that $F_g$ is a bounded linear functional on $A^q\ltt\Om\rtt$. Thus, by Lemma~\ref{le:dual of Ap}, there exists a $\phi_g\in A^p\ltt\Om\rtt$, such that 
\[
F_g(\psi)=\int_\Om \overline {\phi_g\ltt\zeta\rtt} \psi\ltt\zeta\rtt dV\ltt\zeta\rtt.
\]
Now, for a fixed $z\in\Om^*$, taking $\psi(\zeta)=\frac{n!}{\pi^n}\ltt1-\left\langle\zeta,z\right\rangle\rtt^{-n-1}$, we obtain that
\bea
\nonumber\mathcal{F}_{n+1}(\phi_g)(z)=\frac{n!}{\pi^n}\int_\Om \frac{\overline{\phi_g\ltt\zeta\rtt}}{\ltt1-\left\langle\zeta,z\right\rangle\rtt^{n+1}} dV\ltt\zeta\rtt&=& \frac{1}{(2\pi i)^n}\int_\Om \dfrac{\ltt g\circ T_\Om\rtt\ltt\zeta\rtt}{\ltt1-\left\langle\zeta,z\right\rangle\rtt^{n+1}}\ltt\mathfrak H\circ T_\Om\rtt\ltt\zeta\rtt dV(\zeta)\\
\nonumber(\text{since }T_\Om^{-1}=T_{\Om^*})\quad 
&=&\frac{1}{(2\pi i)^n}\int_\Om \dfrac{g(T_\Om(\zeta))\mathfrak H(T_\Om(\zeta))}{\ltt1-\left\langle T_{\Om^*}(T_\Om(\zeta)),z\right\rangle\rtt^{n+1}} \frac{h_\Om(\zeta)}{h_\Om(T_{\Om^*}(T_\Om(\zeta)))}dV(\zeta)\\
\nonumber(\text{by Lemma \ref{le:cov dom}})\quad
&=&\frac{1}{(2\pi i)^n}\int_{\Om^*} \dfrac{g(\eta)\mathfrak H(\eta)}{\ltt1-\left\langle T_{\Om^*}(\eta),z\right\rangle\rtt^{n+1}} \frac{1}{h_\Om(T_{\Om^*}(\eta))} dV(\eta)\\
\nonumber(\text{by \eqref{eq:def of FLD kera}, \eqref{eq:tildeK}, and \eqref{eq:mathfrak H}})\quad
&=&\frac{1}{(2\pi i)^n}\int_{\Om^*}\dfrac{g\ltt\eta\rtt}{\widetilde{K}_{\Om^*}\ltt\eta,z\rtt^{n+1}}\mathfrak h_{\Om^*}(\eta) dV(\eta)\\
\nonumber(\text{by \eqref{eq:comp vol}}) \quad 
&=& \frac{1}{(2\pi i)^n}\int_{\Om^*}g(\eta)\ltt\overline\partial_{\eta} G \rtt^n(\eta,z)\\
(\text{by Lemma~\ref{le:repr}})\quad \label{eq:repr fant}&=& g(z).
\eea
This completes the proof of Theorem~\ref{th:main1}.
\begin{section}{Proof of Theorem~\ref{th:PW bergman}}
We begin by proving that $\mathfrak B_n$ is an normed space isomorphism between $A^2(\C^n,\omega_\Om)$, and $A^2(\Om^*)$.
\begin{lemma}\label{le:borel bergman}
 
    Let $\Om\subset\C^n$ be a domain which satisfies the hypothesis of Theorem~\ref{th:PW bergman}.  Then,
    \bea\label{eq:norm distor borel}
   C_1 \|f\|_{A^2\ltt\C^n,\omega_\Om\rtt} \leq \|\mathfrak B_n(f)\|_{A^2(\Om^*)}\leq  C_2 \|f\|_{A^2\ltt\C^n,\omega_\Om\rtt},\quad f\in A^2\ltt\C^n,\omega_\Om\rtt,
    \eea
    where $C_1,C_2$ are positive constants not depending on $f$.
\end{lemma}
\begin{proof}
Let $F\in A^2(\C^n,\omega_\Om)$. For a fixed $\zeta\in \Om^*$, consider the sets
\[
\Om^*_\zeta=\left\{\eta\in\C:\eta\zeta\in \Om^*\right\},\quad M_\zeta=\ltt \Om^*_\zeta\rtt^*,
\]
and the function 
\[
F_\zeta(w)=F(w\zeta) w^{n-1},\quad w\in\C.
\]
Following the same computations as in \cite{Li02}, we obtain the following.
\begin{enumerate}
       \item[(i)] If $\varphi$ is a $\sigma_{\Om^*}$-measurable function on $b\Om^*$, then
    \be\label{eq:bdry int equiv}
    \int_{b\Om^*}\varphi(\zeta)d\sigma_{\Om^*}(\zeta)\approx\int_{b\Om^*}\int_{b\Om^*_\zeta}\varphi(\lambda\zeta)d\sigma_{\Om^*_\zeta}(\lambda) d\sigma_{\Om^*}(\zeta).
    \ee
     See \cite[Lemma~23]{Li02}.
    \item[(ii)] For any $F\in A^2(\C^n,\omega_\Om) $,
    \bea\label{eq:eqiv norm PW aux}
    \|F\|^2_{A^2(\C^n,\omega_\Om)}\approx\int_{b\Om^*}\|F_\zeta\|^2_{A^2\ltt\C,\omega_{M_\zeta}\rtt}d\sigma_{\Om^*}(\zeta).
    \eea
    See \cite[pages~308-309]{Li02}.
    \item[(iii)] $M_\zeta$ is a bounded $\cont^2$-smooth strongly convex domain in $\C$, and 
    \be\label{eq:support fns reln}
    H_{M_\zeta}(w)=H_\Om(w\zeta),\quad w\in\C.
    \ee
    See \cite[Lemma~10]{Li02}.
    \item[(iv)]   
    The function $F_\zeta$ satisfies
    \be\label{eq:one d PW}
    \int_{\C}|F_\zeta(w)|^2 |w|^{\frac{3}{2}}e^{-2H_{M_\zeta}(w)} \ltt dd^cH_{M_\zeta}\rtt(w)<\infty.
    \ee
    See \cite[page~309]{Li02}.
    \item [(v)]
    $F$ satisfies \eqref{eq:Borel dfn cond}, and hence, $\mathfrak B_n F$ is well-defined on $\Om^\circ$, the polar set of $\Om$, and has an analytic continuation to $\Om^*$. See \cite[ Theorem~17]{Li02}.
\end{enumerate}
For $\eta>0$,
\beas
\eta^{(n+1)}\mathfrak B_n F\ltt\eta\zeta\rtt&=&\eta^{(n+1)}\int_0^\infty F\ltt{t\eta\zeta}\rtt t^n e^{-t}dt\\
&=&\eta^2\int_0^\infty F_{\zeta}\ltt{t\eta}\rtt{t} e^{-t}{dt},\\
&=& \eta^{2}\mathfrak B_1 F_\zeta\ltt {\eta}\rtt.
\eeas
Note that the LHS is convergent for all $\eta$ satisfying ${\eta}\zeta\in \Om^\circ$. Equivalently by \eqref{eq:support fns reln}, {$H_{\Om}(\eta\zeta)=H_{{M_\zeta}}(\eta)<1$, i.e., $\eta\in {M_\zeta}^{\circ} $}. Similar to $\mathfrak B_nF$ , $\mathfrak B_1 F_\zeta$ has a holomorphic extension to 
$ {M_\zeta}^*=\Om^*_\zeta$. From \eqref{eq:one d PW} it follows that, $F_\zeta\in A^2\ltt\C,\omega_{{M_\zeta}}\rtt$, and hence, applying Lemma~\ref{le:borel one d} we get that
\beas
\|F_\zeta\|^2_{A^2\ltt\C,\omega_{M_\zeta}\rtt}&\approx&  \left\|\mathfrak B_1F_\zeta\right\|^2_{A^2(\Om^*_\zeta)}\\
&=&\int_{\Om^*_\zeta}\left|\mathfrak B_n F_\zeta(\eta\zeta)\right|^2|\eta|^{2n-2} dV(\eta)\\
&\approx& \int_0^1\int_{b\Om^*_\zeta} \left|\mathfrak B_n F_\zeta(r\lambda\zeta)\right|^2 r^{2n-1} d\sigma_{\Om^*_\zeta}(\lambda) dr.
\eeas
Now integrating both side with respect to $\zeta$ on $b\Om^*$, and applying \eqref{eq:eqiv norm PW aux} we get that
\bea
\nonumber\|F\|^2_{A^2(\C^n,\omega_\Om)}&\approx&\int_{b\Om^*}\|F_\zeta\|^2_{A^2\ltt\C,\omega_{M_\zeta}\rtt} d\sigma_{\Om^*}(\zeta)\\
\nonumber&\approx& \int_0^1\int_{b\Om^*}\int_{b\Om^*_\zeta} \left|\mathfrak B_n F_\zeta(r\lambda\zeta)\right|^2 r^{2n-1} d\sigma_{\Om^*_\zeta}(\lambda) d\sigma_{\Om^*}(\zeta)dr\\
\nonumber(\text{by}\,\eqref{eq:bdry int equiv})&\approx&\int_0^1\int_{b\Om^*} \left|\mathfrak B_n F(r\zeta)\right|^2 r^{2n-1} d\sigma_{\Om^*}(\zeta)dr\\
\label{eq:borel norm equiv}(\text{by}\,\eqref{eq:polar coordinate})&\approx& \|\mathfrak B_n(F)\|^2_{A^2(\Om^*)}.
\eea
This shows \eqref{eq:norm distor borel}, and hence, Lemma~\ref{le:borel bergman} is proved.
\end{proof}
Let $f\in A^2\ltt \Om\rtt$. Then by definition, for any $z\in\C^n$,
\beas
\left|\mathcal L(f)(z)\right|&\leq&\int_\Om |f(\zeta)| e^{\rea \langle\zeta,z\rangle} dV(\zeta),\\
&\leq& e^{H_\Om(z)} \int_\Om |f(\zeta)| dV(\zeta), \\
&\leq& C e^{H_\Om(z)},
\eeas
where $C$ is a positive constant does not depends on $f$ or $z$.
This shows $\mathcal L(f)$ satisfies the condition in  \eqref{eq:Borel dfn cond}, and hence, $\mathcal L(f)\in \hol_{\operatorname{exp}}(\Om)$. Thus, $\mathfrak B_n\ltt\mathcal L(f)\rtt$ has a holomorphic extension to $\Om^*$. Furthermore, for $z\in\Om^*$,
\bea\label{eq:Lap borel reln}
\nonumber\mathfrak B_n\ltt\mathcal L(f)\rtt(z)&=&\int_0^\infty \int_{\Om} \overline{f(\zeta)} e^{-t\ltt-\langle\zeta,z\rangle+1\rtt} t^n dV(\zeta) dt\\
\nonumber&=& \int_\Om \overline{f(\zeta)}\int_0^\infty  e^{-t\ltt-\langle\zeta,z\rangle+1\rtt} t^n dt  dV(\zeta)\\
\nonumber&=&n!\int_\Om\frac{\overline{f(\zeta)}}{\ltt 1-\langle\zeta,z\rangle\rtt^{n+1}} dV(\zeta)\\
&=& \pi^n \F_{n+1}(f)(z).
\eea
 By Theorem~\ref{th:main1}, $\F_{n+1}$ is a normed space isomorphism between $A^2(\Om)$ and $A^2(\Om^*)$. Thus, proof of the theorem reduces to showing $\mathfrak B_n$ is a normed space isomorphsim between $A^2(\C^n,\omega_\Om)$, and $A^2(\Om^*)$. However, from Lemma~\ref{le:borel bergman}, we have already obtained the injectivity and the $L^2$-boundedness of $\mathfrak B_n$. Hence, to conclude the theorem, it suffices to show that $\mathfrak B_n$ is surjective.

Let $g\in A^2(\Om^*)$. Consider the following analytic functional on $\hol'\ltt\overline{\Om}\rtt$.
\[
\mu_g(\psi)=\int_\Om \psi(\zeta)\ltt g\circ T_\Om\rtt(\zeta) \ltt\mathfrak H\circ T_\Om\rtt(\zeta)  dV(\zeta),\quad\forall \psi\in \hol(\overline \Om),
\]
where $\mathfrak H:\Om^*\rightarrow\C$ is defined as 
\bes\label{eq:pairing}
\mathfrak H\ltt\eta\rtt= {\mathfrak h_{\Om^*}(\eta)}{h_{\Om}\ltt T_{\Om^*}\ltt\eta\rtt\rtt}\ltt\frac{1+\rho_{\Om^*}(\eta)}{\left\langle\partial \rho_{\Om^*}(\eta),\eta\right\rangle}\rtt^{n+1},
\ees
for $\mathfrak h_{\Om^*}$, $ h_{\Om}$ as in \eqref{eq:comp vol} and  \eqref{eq:cov}, respectively, and $T_\Om, T_{\Om^*}$ as in \eqref{eq:cov map}. 
Let $\widehat\mu_g$ be the Laplace transform of $\mu_g$ defined as
\[
\widehat\mu_g(z)=\mu_g\ltt e^{\langle\cdot,z\rangle}\rtt,\quad z\in\C^n.
\]
Then, $\widehat\mu_g\in\hol(\C^n)$, and satisfies the condition in \eqref{eq:Borel dfn cond}. Hence, $\mathfrak B_n\ltt\widehat \mu_g\rtt$ has an analytic continuation to $\Om^*$.
\beas
\mathfrak B_n\ltt\widehat \mu_g\rtt=\pi^n{\F_{n+1}(\mu_g)}.
\eeas

 Following the same computation as in the last paragraph of the previous section, we get that
 \[
 \F_{n+1}(\mu_g)(z)=g(z),\quad z\in \Om^*.
 \]
Also, as $g\in A^2(\Om^*)$, using the same computations in \eqref{eq:borel norm equiv} backwards, we get, $\widehat \mu_g\in A^2(\C^n,\omega_\Om)$. This shows $\frac{\widehat \mu_g}{\pi^n}$ is in the pre-image of $g$ under $\mathfrak B_n$, which shows $\mathfrak B_n$ is surjective, and the proof is complete.   
\end{section}
\begin{section}{Proof of Theorem~\ref{th:counter}}
{
Let us recall \beas
\Om=\left\{(\zeta_1,\zeta_2)\in\C^2: |\zeta_1|+|\zeta_2|<1\right\}.
\eeas
By definition, $\Om$ is a circled domain, in fact, it is a Reinhardt domain. Consequently, it follows from \cite[Proposition~18]{Li02} that (the interior of) $\Om^*$ can be given as }
   \[
   \Om^*=\left\{z\in\C^2: H_\Om(z)<1\right\},
   \]
   {where $H_\Om(\overline z)$ is the support function of $\Om$.} In this case, $H_\Om(z)=\max\{|z_1|,|z_2|\}$. Hence,
$$\Om^*=\left\{\ltt z_1,z_2\rtt\in\C^2: |z_1|<1,|z_2|<1\right\}.$$
Consider the following parameterization of $\Om$. 
\bea
\vartheta:(0,1)\times[0,1]\times [0,2\pi)^2&\rightarrow& \Om,\notag \\
      (r,s,\theta_1,\theta_2)&\mapsto&\ltt rse^{i\theta_1},r(1-s)e^{i\theta_2}\rtt.
      \label{eq:param of diamond}
      \eea
      Under this change of co-ordinates, the pull-back of the Lebesgue measure is given by
      \bea\label{eq:pullback of lbge}
      \vartheta^*\ltt dV\rtt=-\frac{1}{4}\vartheta^*\ltt dz_1 dz_2 d\overline{z_1} d\overline{z_2}\rtt=\frac{1}{2}r^3s(1-s) dr ds d\theta_1 d\theta_2.
      \eea
    Since $\Om$ is a complete Reinhardt domain, any holomorphic function has a global power series expansion on $\Om$. 
    \begin{lemma}\label{le:power series}
Let $\Om=\{(\zeta_1,\zeta_2)\in\C^2: |\zeta_1|+|\zeta_2|<1\}$.  
Then, the following holds.
\begin{itemize}
    \item[(i)] 
$f\in A^2\ltt\Om\rtt$ if and only if
    \be\label{eq:series exp diamond}
    f(\zeta_1,\zeta_2)=\sum\limits_{(m_1,m_2)\in \N^2} a_{m_1m_2}\zeta_1^{m_1}\zeta_2^{m_2},\quad \ltt\text{in} \,L^2\ltt\Om\rtt\rtt,
    \ee
       where $\ltt a_{m_1,m_2}\rtt_{\N^2}$ are complex numbers that satisfy
       \be\label{eq:condition on coeff diamond}
       \sum_{(m_1,m_2)\in \N^2}\frac{|a_{m_1,m_2}|^2}{(m_1+m_2+2)}\frac{{\ltt 2m_1+1\rtt!}{\ltt 2m_2+1\rtt!}}{{\ltt 2m_1+2m_2+3\rtt!}}<\infty.
       \ee
       Moreover,
       \be\label{eq:norm eq diamond}
       \|f\|_{A^2\ltt\Om\rtt}^2\approx \sum_{(m_1,m_2)\in \N^2}\frac{|a_{m_1,m_2}|^2}{(m_1+m_2+2)}\frac{{\ltt 2m_1+1\rtt!}{\ltt 2m_2+1\rtt!}}{{\ltt 2m_1+2m_2+3\rtt!}}.
       \ee
       \item[(ii)]
       $F\in A^2\ltt\Om^*\rtt$ if and only if
    \be\label{eq:series exp polydisc}
    F(z_1,z_2)=\sum\limits_{(m_1,m_2)\in \N^2} t_{m_1m_2}z_1^{m_1}z_2^{m_2},\quad \ltt\text{in} \,L^2\ltt\Om^*\rtt\rtt,
    \ee
       where $\ltt t_{m_1,m_2}\rtt_{\N^2}$ are complex numbers that satisfy
       \be\label{eq:condition on coeff polydisc}
       \sum_{(m_1,m_2)\in \N^2}\frac{|t_{m_1,m_2}|^2}{(m_1+1)(m_2+1)}<\infty.
       \ee
        Moreover,
       \be\label{eq:norm eq polydisc}
       \|F\|_{A^2\ltt\Om^*\rtt}^2\approx  \sum_{(m_1,m_2)\in \N^2}\frac{|t_{m_1,m_2}|^2}{(m_1+1)(m_2+1)}.
       \ee
       \item[(iii)] Let $G\in \hol(\C^2)$ with the power series expansion
       \be\label{eq:series exp PW}
    G(z_1,z_2)=\sum\limits_{(m_1,m_2)\in \N^2} \ell_{m_1m_2}z_1^{m_1}z_2^{m_2},\quad \ltt \text{uniformly on compact subsets of } \C^2
    \rtt.
    \ee
    \begin{itemize}
        \item [(a)]
     $G\in A^2(\C^2,\omega_\Om)$ if and only if
       the complex numbers $\ltt \ell_{m_1,m_2}\rtt_{\N^2}$ satisfy
       \be\label{eq:condition on coeff PW}
       \sum_{(m_1,m_2)\in \N^2}{|\ell_{m_1,m_2}|^2}\frac{\Gamma(2m_1+2m_2+\frac{7}{2})}{2^{2m_1+2m_2}}<\infty.
       \ee
        Moreover, 
       \be\label{eq:norm eq PW}
       \|G\|_{A^2\ltt\C^2,\omega_\Om\rtt}^2\approx  \sum_{(m_1,m_2)\in \N^2}{|\ell_{m_1,m_2}|^2}\frac{\Gamma(2m_1+2m_2+\frac{7}{2})}{2^{2m_1+2m_2}}.
       \ee
      {
       \item[(b)] $G\in A^2(\C^n,\mu_\Om)$ if and only if the complex numbers $\ltt \ell_{m_1,m_2}\rtt_{\N^2}$ satisfy
       \be\label{eq:norm condition on co eff PW YL}
       \sum_{(m_1,m_2)\in \N^2}{|\ell_{m_1,m_2}|^2}\frac{(2m_1+2m_2+2)!}{2^{2m_1+2m_2}}<\infty.
       \ee
        Moreover, 
       \be\label{eq:norm eq PW YL}
       \|G\|_{A^2\ltt\C^2,\mu_\Om\rtt}^2\approx  \sum_{(m_1,m_2)\in \N^2}{|\ell_{m_1,m_2}|^2}\frac{(2m_1+2m_2+2)!}{2^{2m_1+2m_2}}.
       \ee
       }
       \end{itemize}
       \end{itemize}
       \end{lemma}
       \begin{proof} Let $f\in A^2(\Om)$. Due to the Reinhardtness of $\Om$, there exists a sequence $(a_{m_1,m_2})_{\N^2}$ such that
       \[
       p_k(\zeta_1,\zeta_2)\rightarrow f(\zeta_1,\zeta_2)\quad \text{in}\, L^2(\Om),
       \]
       where $p_k(\zeta_1,\zeta_2)=\sum\limits_{m_1,m_2=0}^k a_{m_1,m_2} \zeta_1^{m_1}\zeta_2^{m_2}$. Now
       \bea\nonumber
       \|p_k\|_{A^2(\Om)}^2&=&\int_\Om\left|{\sum\limits_{m_1,m_2=0}^k a_{m_1,m_2}} \zeta_1^{m_1}\zeta_2^{m_2}\right|^2 dV(\zeta_1,\zeta_2)\\
       \nonumber
       \text{(by \eqref{eq:param of diamond}, \eqref{eq:pullback of lbge})}&=& 4\pi^2 \sum\limits_{m_1,m_2=0}^k |a_{m_1,m_2}|^2\int_0^1r^{2m_1+2m_2+3}dr\int_0^1 s^{2m_1+1}(1-s)^{2m_2+1} ds\\
       \nonumber&\approx& \sum\limits_{m_1,m_2=0}^k \frac{|a_{m_1,m_2}|^2}{m_1+m_2+2}\frac{(2m_1+1)!(2m_2+1)!}{(2m_1+2m_2+3)!}
       \eea
       As $\lim\limits_{k\rightarrow\infty}\|p_k\|_{A^2(\Om)}^2=\|f\|_{A^2(\Om)}^2$, it follows that \eqref{eq:condition on coeff diamond}, and \eqref{eq:norm eq diamond} holds. 

       Conversely, if $(a_{m_1,m_2})_{\N^2}$ is a sequence of complex numbers satisfying \eqref{eq:condition on coeff diamond}, then the function $f$ defined as in \eqref{eq:series exp diamond} is in $L^2(\Om)$, and is approximable by holomorphic polynomials of the form  $p_k(\zeta_1,\zeta_2)=\sum\limits_{m_1,m_2=0}^ka_{m_1,m_2}\zeta_1^{m_1} \zeta_2^{m_2}$, $k\in\N$. Thus, $f\in A^2(\Om)$. This completes the proof of (i).

       Suppose $q_k$ is a polynomial of the form $q_k(z_1,z_2)=\sum\limits_{m_1,m_2=0}^kt_{m_1,m_2}z_1^{m_1} z_2^{m_2}$. Then 
       \bea\nonumber
       \|q_k\|_{A^2(\Om^*)}^2&=&\int_{\Om^*}\left|{\sum\limits_{m_1,m_2=0}^k t_{m_1,m_2}} z_1^{m_1}z_2^{m_2}\right|^2 dV(z_1,z_2)\\
       \nonumber
       &=& 4\pi^2 \sum\limits_{m_1,m_2=0}^k |t_{m_1,m_2}|^2\int_0^1r^{2m_1+1}dr\int_0^1 r^{2m_2+1} dr\\
      \nonumber &\approx& \sum\limits_{m_1,m_2=0}^k \frac{|t_{m_1,m_2}|^2}{(m_1+1)(m_1+1)}.
       \eea
Now, following the same argument as in (i), we get (ii).

From \cite[Theorem 1]{BS08}, it follows that the measure $(dd^c H_\Om)^2$ is supported on the real hypersurface $M=\{(z_1,z_2)\in\C^2:|z_1|=|z_2|\}$, and if we parameterize $M$ via $\vartheta_M
         \ltt r,\psi_1,\psi_2\rtt=\ltt re^{i\psi_1},re^{i\psi_2}\rtt$, we obtain that
        \be\label{eq:MA bi disc}
       \vartheta_M^* (dd^cH_\Om)^2\ltt r,\psi_1,\psi_2\rtt\approx dr d\psi_1 d\psi_2,\quad (r,\psi_1,\psi_2)\in(0,\infty)\times[0,2\pi)^2.
        \ee
        For an explicit computation, see \cite[Section~6.2(iv)]{AC24}. 
Next, if we take a polynomial of the form $\mathcal P_k(z_1,z_2)=\sum\limits_{m_1,m_2=0}^k\ell_{m_1,m_2}z_1^{m_1} z_2^{m_2}$. Then from \eqref{eq:MA bi disc}, it follows that
 \bea\nonumber
       \|\mathcal P_k\|_{A^2(\C^2,\omega_\Om)}^2&=&\int_{\C^2}\left|{\sum\limits_{m_1,m_2=0}^k \ell_{m_1,m_2}} z_1^{m_1}z_2^{m_2}\right|^2 e^{-2\max\{|z_1|,|z_2|\}}\|z\|^{\frac{5}{2}}(dd^cH_\Om)^2(z_1,z_2)\\
       \nonumber
       &\approx&  \sum\limits_{m_1,m_2=0}^k |\ell_{m_1,m_2}|^2\int_0^\infty r^{2m_1+2m_2+\frac{5}{2}}e^{-2r} dr\\
      \nonumber &\approx& \sum\limits_{m_1,m_2=0}^k \frac{|\ell_{m_1,m_2}|^2}{2^{2m_1+2m_2}}\Gamma\ltt 2m_1+2m_2+\frac{7}{2}\rtt.
       \eea
Now, following the same steps as in (i) and (ii), we obtain (iii)(a).
{
To show (iii)(b), we first show that
for any $z\in\C^n$, satisfying $H_\Om(z)>1$,
\[
\left\|e^{\langle\cdot, z\rangle}\right\|^2_{L^2(\Om)}\approx {e^{2H_\Om(z)}}\|z\|^{-2}.
\]
Parameterizing the set $\{z\in\C^2: H_\Om(z)>1\}$ as $\{(te^{i\psi_1},te^{i\psi_2}): t>1, \psi_1,\psi_2\in[0,2\pi]\}$ we get
\begin{align}\label{eq:main est counter}
\nonumber e^{-2H_\Om(z)}&\|z\|^2\left\|e^{\langle\cdot, z\rangle}\right\|^2_{L^2(\Om)}\\ 
\nonumber&\approx t^2\int_0^1\int_0^1\int_0^{2\pi}\int_0^{2\pi} e^{2tr\ltt s\cos(\theta_1-\psi_1)+(1-s)\cos(\theta_2-\psi_2)-1\rtt} e^{2t(r-1)}r^3s(1-s) d\theta_1 d\theta_2 ds dr\\
&=t^2\int_0^1\int_0^1\int_0^{2\pi}\int_0^{2\pi} e^{2tr\ltt s\cos\theta_1+(1-s)\cos\theta_2-1\rtt} e^{2t(r-1)}r^3s(1-s) d\theta_1 d\theta_2 ds dr.
\end{align}
For $t>1$ and $r>0$, let us denote
\begin{align*}
Q(r,t)&=rt\int_0^1\int_0^{2\pi}\int_0^{2\pi} e^{2rt\ltt s\cos\theta_1+(1-s)\cos\theta_2-1\rtt} s(1-s) d\theta_1 d\theta_2 ds\\
&= 
4\pi^2 e^{-2rt}rt\int_0^1 I_0(2rts) I_0(2rt(1-s)) s(1-s) ds, 
\end{align*}
where $I_0$ is the modified Bessel function of the first kind of order $0$, given by
\[
I_0(x)=\frac{1}{2\pi}\int_0^{2\pi} e^{{x} \cos\theta} d\theta,\quad \text{for } x\in\R.
\]
By \cite[9.7.1]{AB48}, as $x\rightarrow\infty$,
\[
I_0(x) \frac{(1+x)^{\frac{1}{2}}}{e^x}\rightarrow 1.
\]
This implies that for $x\geq 0$, 
\[
I_0(x)\approx  \frac{e^x}{(1+x)^{\frac{1}{2}}}.
\]
Using this estimate on $Q(r,t)$, we obtain that for $t>1$ and $r>0$,
\beas
Q(r,t)&\approx rt\int_0^1 (1+2rts)^{-\frac{1}{2}}(1+2rt(1-s))^{-\frac{1}{2}}  s(1-s) ds\\
&= rt\int_0^1 (1+2rt+(2rt)^2 s(1-s))^{-\frac{1}{2}}  s(1-s) ds.
\eeas
From this expression, for $t>1$ and $r>0$ such that $rt\leq 1$,
\beas
Q(r,t)\approx rt \int_0^1 s(1-s) ds\approx rt.
\eeas
Furthermore, for $t>1$ and $r>0$ such that $rt\geq 1$,
\begin{align*}
\nonumber Q(r,t)&\approx \int_0^1  \ltt s(1-s)+\frac{1}{rt}\rtt^{-\frac{1}{2}}  s(1-s) ds\\
&\approx \int_0^1 s^{\frac{1}{2}}(1-s)^{\frac{1}{2}} ds\approx 1.
\end{align*}
Consequently, as $t\rightarrow\infty$,
 \begin{align}\label{eq:Qt 1}
\nonumber t\int_0^{\frac{1}{t}} e^{2t(r-1)}Q(r,t) r^2 dr&\leq t^2 e^{2-2t}\int_0^{\frac{1}{t}}r^3 dr\\
&= \frac{ e^{2-2t}}{4t^2}\rightarrow 0.
 \end{align}
 Furthermore, by applying a change of variable $u=t(1-r)$, and applying the dominated convergence theorem, as $t\rightarrow \infty$,
 \begin{align}\label{eq:Qt 2}
    \nonumber t \int_{\frac{1}{t}}^1 e^{2t(r-1)}Q(r,t) r^2 dr&= t \int_{\frac{1}{t}}^1 e^{2t(r-1)} r^2 dr\\
     &=\int_0^{t-1} e^{-2u} \ltt1-\frac{u}{t}\rtt^2 du \rightarrow \int_0^\infty e^{-2u} du=\frac{1}{2}.
 \end{align}
 Combining \eqref{eq:Qt 1} and \eqref{eq:Qt 2}, we get that as $t\rightarrow\infty$,
 \bea\label{eq:Qt final}
 t\int_0^1 e^{2t(r-1)}Q(r,t) r^2 dr\rightarrow \frac{1}{2}.
 \eea
 Finally by \eqref{eq:main est counter} and \eqref{eq:Qt final}, for $t>1$, 
 \beas
 e^{-2H_\Om(z)}\|z\|^2\left\|e^{\langle\cdot, z\rangle}\right\|^2_{L^2(\Om)}\approx t\int_0^1 e^{2t(r-1)}Q(r,t) r^2 dr\approx 1.
 \eeas
 This completes the proof of our claim.

Combining this with \eqref{eq:MA bi disc}, we get
\bea\nonumber
       \|\mathcal P_k\|_{A^2(\C^2,\mu_\Om)}^2&\approx&\int_{\C^2}\left|{\sum\limits_{m_1,m_2=0}^k \ell_{m_1,m_2}} z_1^{m_1}z_2^{m_2}\right|^2 e^{-2\max\{|z_1|,|z_2|\}}\|z\|^{2}(dd^cH_\Om)^2(z_1,z_2)\\
       \nonumber
       &\approx&  \sum\limits_{m_1,m_2=0}^k |\ell_{m_1,m_2}|^2\int_0^\infty r^{2m_1+2m_2+2}e^{-2r} dr\\
      \nonumber &\approx& \sum\limits_{m_1,m_2=0}^k \frac{|\ell_{m_1,m_2}|^2}{2^{2m_1+2m_2}}(2m_1+2m_2+2)!.
       \eea
 Now, following the same steps as in (i) and (ii), we obtain (iii)(b).}
\end{proof}

      Next, we expand the Fantappi{\`e} and Laplace transforms of $A^2\ltt\Om\rtt$-functions in terms of power series. 
      \begin{lemma}\label{le:Fan series}
          Let $\Om$ be as above, and $f\in A^2\ltt\Om\rtt$, which admits the expansion in \eqref{eq:series exp diamond}. 
          \begin{itemize}
              \item[(i)] 
          The Fantappi{\`e} transform of $f$ is given by,
      \bea\label{eq:Fan expand}
      \F_3\ltt f\rtt(z)=\sum\limits_{(m_1,m_2)\in\N^2} t_{m_1,m_2} z_1^{m_1}z_2^{m_2},\quad  z=(z_1,z_2)\in\Om^*,
      \eea
      where 
      \bea\label{eq:Fan coeff}
      t_{m_1,m_2}=2\overline{a_{m_1,m_2}}\frac{\ltt m_1+m_2+1 \rtt!}{m_1! m_2!}\dfrac{(2m_1+1)!(2m_2+1)!}{(2m_1+2m_2+3)!}.
      \eea
      \item [(ii)]
      The Laplace transform of $f$ is given by,
      \bea\label{eq:Lap expand}
      \mathcal L\ltt f\rtt(z)=\sum\limits_{(m_1,m_2)\in\N^2} \ell_{m_1,m_2} z_1^{m_1}z_2^{m_2},\quad  z=(z_1,z_2)\in\C^n,
      \eea
      where 
      \bea\label{eq:Lap coeff}
      \ell_{m_1,m_2}=\pi^2\frac{\overline{a_{m_1,m_2}}}{m_1! m_2!}\dfrac{(2m_1+1)!(2m_2+1)!}{(m_1+m_2+2)(2m_1+2m_2+3)!}.
      \eea

      \end{itemize}
      \end{lemma}
      \begin{proof}
        Let $z\in\Om^*$. By the Cauchy--Schwarz inequality, there exists a $C_z>0$ such that,
        \beas
        \left|\F_3\ltt f\rtt(z)\right|\leq C_z \|f\|_{A_2\ltt\Om\rtt},
        \eeas
        and
        \beas
         \left|\mathcal L\ltt f\rtt(z)\right|\leq C_z \|f\|_{A_2\ltt\Om\rtt},\quad f\in A^2(\Om)
        \eeas
        Thus, it suffices to establish the claim for polynomials of the form 
    $
    p_k\ltt z_1,z_2\rtt=\sum\limits_{m_1,m_2=0}^k a_{m_1,m_2}z_1^{m_1}\,z_2^{m_2}.
    $
    We first begin by showing that, $\left|\left\langle\zeta,z\right\rangle\right|< 1,\,\forall\zeta\in\overline{\Om}$. To see this, if there exists $\widetilde\zeta\in\overline{\Om}$ such that $\left|\left\langle\widetilde \zeta,z\right\rangle\right|=1$. Since $\Om$ is Reinhardt, $\frac{\widetilde{\zeta}}{\left\langle\widetilde \zeta,z\right\rangle}\in\overline \Om$, and $\left\langle\frac{\widetilde{\zeta}}{\left\langle\widetilde \zeta,z\right\rangle},z\right\rangle=1$, which is not possible since, $z\in \Om^*$, which proves our claim. Thus, using the series  of $\ltt1-\left\langle\zeta,z\right\rangle\rtt^{-3}$, and \eqref{eq:series exp diamond}, we get that for each $z=(z_1,z_2)\in\Om^*$,
           \beas
      \mathcal F_3\ltt p_k\rtt(z)&=&\frac{2}{\pi^2}\sum_{(k_1,k_2)\in\N^2} \dfrac{\ltt k_1+k_2+2\rtt!}{k_1!k_2!} \ltt \int_\Om \overline {f\ltt\zeta\rtt} \zeta_1^{k_1}\zeta_2^{k_2} dV(\zeta)\rtt z_1^{k_1} z_2^{k_2}\\ 
      &=&\frac{2}{\pi^2}\sum_{(k_1,k_2)\in\N^2}\sum\limits_{m_1,m_2=0}^k \dfrac{\ltt k_1+k_2+2\rtt!}{k_1!k_2!} \ltt \overline{a_{m_1,m_2}}\int_\Om \overline{\zeta_1^{m_1}}\overline{\zeta_2^{m_2}}\zeta_1^{k_1}\zeta_2^{k_2} dV(\zeta)\rtt z_1^{k_1} z_2^{k_2}.
      \eeas
      Similarly for $\mathcal L$, we get that for each $z=(z_1,z_2)\in\C^2$,
      \beas
      \mathcal L(p_k)(z)&=&\sum\limits_{(k_1,k_2)\in\N^2} \ltt \int_\Om \overline {f\ltt\zeta\rtt} \zeta_1^{k_1}\zeta_2^{k_2} dV(\zeta)\rtt\frac{z_1^{k_1}z_2^{k_2}}{k_1!k_2!}\\
      &=&\sum\limits_{(k_1,k_2)\in\N^2}\ltt \overline{a_{m_1,m_2}}\int_\Om \overline{\zeta_1^{m_1}}\overline{\zeta_2^{m_2}}\zeta_1^{k_1}\zeta_2^{k_2} dV(\zeta) \rtt\frac{z_1^{k_1}z_2^{k_2}}{k_1!k_2!}. 
      \eeas
     {Finally, by combining the orthogonality of monomials (due to Reinhardtness) with the computation in the proof establishing \eqref{eq:condition on coeff diamond}, we deduce that}
      \beas
      \int_{\Om}\overline{\zeta_1^{m_1}}\overline{\zeta_2^{m_2}}\zeta_1^{k_1}\zeta_2^{k_2} dV(\zeta)=\begin{cases}
         \frac{2\pi^2}{\ltt 2m_1+2m_2+4\rtt}\frac{{\ltt 2m_1+1\rtt!}{\ltt 2m_2+1\rtt!}}{{\ltt 2m_1+2m_2+3\rtt!}},&\quad \text{for}\,(k_1,k_2)=(m_1,m_2),\\
         0,&\quad \text{otherwise}.
      \end{cases}
      \eeas
      Replacing this in the above equation, the claim follows for $p_k$, and hence for $f\in A^2\ltt\Om\rtt$.
       \end{proof}

       Let $f\in A^2\ltt\Om\rtt$, with power series as in \eqref{eq:series exp diamond}.  Due to Lemma~\ref{le:Fan series}, we get that for $z\in\Om^*$, $\F_3\ltt f\rtt(z)=\sum\limits_{(m_1,m_2)\in\N^2}t_{m_1,m_2} z_1^{m_1}z_2^{m_2}$, where
      $t_{m_1,m_2}$ is as in \eqref{eq:Fan coeff}.
    Now, using \eqref{eq:norm eq diamond}, \eqref{eq:norm eq polydisc}, and computations we get that
    \beas
    \|\F_3\ltt f\rtt\|_{A^2\ltt\Om^*\rtt}^2 &\approx&
\sum_{(m_1,m_2)\in \N^2}\frac{|t_{m_1,m_2}|^2}{(2m_1+2)(2m_2+2)}\\
&\approx&\sum_{(m_1,m_2)\in \N^2}\frac{\left|a_{m_1,m_2}\right|^2}{\ltt 2m_1+1\rtt\ltt 2m_2+1\rtt}\frac{{\ltt m_1+m_2+1\rtt!}^2}{{m_1!}^2 {m_2!}^2}\frac{{\ltt 2m_1+1\rtt!}^2{\ltt 2m_2+1\rtt!}^2}{{\ltt 2m_1+2m_2+3\rtt!}^2}
\\
&\lesssim& \sum_{(m_1,m_2)\in \N^2}\frac{|a_{m_1,m_2}|^2}{(m_1+m_2+2)}\frac{{\ltt 2m_1+1\rtt!}{\ltt 2m_2+1\rtt!}}{{\ltt 2m_1+2m_2+3\rtt!}}\\
&\lesssim& \|f\|_{A^2\ltt\Om\rtt}^2,
  \eeas
 where the second last inequality follows from Stirling's approximation applied as follows:
  \beas
\frac{m_1+m_2+2}{\ltt 2m_1+1\rtt\ltt 2m_2+1\rtt}\frac{{\ltt m_1+m_2+1\rtt!}^2}{{m_1!}^2 {m_2!}^2}\frac{{\ltt 2m_1+1\rtt!}{\ltt 2m_2+1\rtt!}}{{\ltt 2m_1+2m_2+3\rtt!}}\approx{m_1^{-\frac{1}{2}}+m_2^{-\frac{1}{2}}} \lesssim 1.
\eeas
This shows that $\F_3$ is a $L^2$-bounded operator from $A^2\ltt\Om\rtt$ to $A^2\ltt\Om^*\rtt$.

A similar computation yields that
\beas
\|\mathcal L (f)\|_{A^2(\C^2,\omega_\Om)}^2&\approx&\sum\limits_{(m_1,m_2)\in\N^2}\frac{|\ell_{m_1,m_2}|^2}{2^{2m_1+2m_2}}{\Gamma\ltt2m_1+2m_2+\frac{7}{2}\rtt}\\
&\approx&\sum\limits_{(m_1,m_2)\in\N^2}{|a_{m_1,m_2}|^2}\dfrac{\Gamma\ltt2m_1+2m_2+\frac{7}{2}\rtt (2m_1+1)!^2(2m_2+1)!^2}{2^{2m_1+2m_2}(m_1+m_2+2)^2m_1!^2m_2!^2(2m_1+2m_2+3)!^2}\\
&\lesssim& \sum_{(m_1,m_2)\in \N^2}\frac{|a_{m_1,m_2}|^2}{(m_1+m_2+2)}\frac{{\ltt 2m_1+1\rtt!}{\ltt 2m_2+1\rtt!}}{{\ltt 2m_1+2m_2+3\rtt!}}\\
&\lesssim& \|f\|_{A^2\ltt\Om\rtt}^2, 
\eeas
where the second last inequality follows from Stirling's approximation applied as follows.
\beas
\dfrac{\Gamma\ltt2m_1+2m_2+\frac{7}{2}\rtt (2m_1+1)!(2m_2+1)!}{2^{2m_1+2m_2}(m_1+m_2+2)m_1!^2m_2!^2(2m_1+2m_2+3)!}&\approx& \frac{m_1^{\frac{1}{2}} m_2^{\frac{1}{2}}}{(m_1+m_2)^{\frac{3}{2}}}\\
\text{(by AM-GM)} &\lesssim& \frac{1}{(m_1+m_2)^{\frac{1}{2}}}\lesssim 1.
\eeas
This shows $\mathcal L$ is $L^2$-bounded operator form $A^2(\Om)$ to $A^2(
\C^2,\omega_\Om)$.

  Next, to show that $\F_3$ is not an onto map, consider $F\ltt z_1,z_2\rtt=\sum\limits_{(m_1,m_2)\in \N^2} t_{m_1,m_2} z_1^{m_1} z_2^{m_2}$, where
 \begin{align*}
      t_{m_1,m_2}=\begin{cases}
         (2k+1)^{\frac{1}{4}},& \text{when } m_1=m_2=k>0,\\
          0,&\text{otherwise}.
      \end{cases}  
    \end{align*}
    Then, by \eqref{eq:norm eq polydisc}, it follows that $\|F\|_{A^2\ltt\Om^*\rtt}^2\approx \sum\limits_{k=1}^\infty \frac{1}{k^{\frac{3}{2}}}<\infty$. Hence, $F\in A^2\ltt\Om^*\rtt$. If there exists a $f\in A^2\ltt\Om\rtt$, such that $\F_3(f)=F$. Then, $f$ admits an expansion, as in \eqref{eq:series exp diamond} for some sequence $\ltt a_{m_1,m_2}\rtt_{\N^2}$ satisfying \eqref{eq:condition on coeff diamond}. Furthermore, from \eqref{eq:Fan coeff} it follows that
    \begin{align*}
      a_{m_1,m_2}=\begin{cases}
        \frac{(2k+1)^{\frac{1}{4}}k!^2(4k+3)!}{2(2k+1)!^3},& \text{when } m_1=m_2=k>0,\\
          0,&\text{otherwise}.
      \end{cases}  
    \end{align*}
    Then, by Stirling's approximation, we get that
    \beas
    \|f\|_{A^2\ltt\Om\rtt}^2&\approx& \sum_{k=1}^\infty \frac{(2k+1)^{\frac{1}{2}}k!^4(4k+3)!}{(2k+1)!^4(2k+2)}\\
&\approx& \sum_{k=1}^\infty \frac{1}{k}=\infty.    
    \eeas
    This is a contradiction, since $f\in A^2\ltt\Om\rtt$, which proves that $\F_3$ is not onto from $A^2\ltt\Om\rtt$ to $A^2\ltt\Om^*\rtt$.
    
    Finally, to show that $\mathcal L$ is not onto, consider $G(z_1,z_2)=\sum\limits_{(m_1,m_2)\in\N^2}\ell_{m_1,m_2}z_1^{m_1}z_2^{m_2}$, where
    \beas
    \ell_{m_1,m_2}=\begin{cases}
        \frac{2^{2k}}{k^{\frac{3}{4}} \Gamma\ltt 4k+\frac{7}{2}\rtt^{\frac{1}{2}}},& \text{when } m_1=m_2=k>0,\\
          0,&\text{otherwise}.
      \end{cases} 
    \eeas
    Then $G\in\hol(\C^2)$, and by \eqref{eq:norm eq PW}, it follows that
    $\|G\|_{A^2(\C^2,\omega_\Om)}^2\approx \sum\limits_{k=1}^\infty \frac{1}{k^\frac{3}{2}}<\infty.$ Hence, $G\in A^2(\C^2,\omega_\Om)$. If there exists a $f\in A^2(\Om)$, such that $\mathcal L(f)=G$, then by the relation \eqref{eq:Lap coeff} we have that
\beas
 a_{m_1,m_2}=\begin{cases}
        \frac{2^{2k} k!^2(2k+2)(4k+3)!}{(2k+1)!^2k^{\frac{3}{4}} \Gamma\ltt 4k+\frac{7}{2}\rtt^{\frac{1}{2}}},& \text{when } m_1=m_2=k>0,\\
          0,&\text{otherwise}.
      \end{cases} 
\eeas
Now, by \eqref{eq:condition on coeff diamond}, and Stirling's approximation, it follows that
\beas
\|f\|^2_{A^2(\Om)}&\approx&\sum\limits_{k=1}^\infty \frac{2^{4k} k!^4(2k+2)(4k+3)!}{(2k+1)!^2k^{\frac{3}{2}} \Gamma\ltt 4k+\frac{7}{2}\rtt}\\
&\approx& \sum\limits_{k=1}^\infty \frac{1}{k}=\infty.
\eeas
This is a contradiction, since $f\in A^2(\Om)$.

{
As for every $(m_1,m_2)\in \N^2$, $(2m_1+2m_2+2)!\leq \Gamma\ltt2m_1+2m_2+\frac{7}{2}\rtt $, by \eqref{eq:condition on coeff PW}, and \eqref{eq:norm condition on co eff PW YL} it follows that
\[
A^2(\C^2,\omega_\Om)\subseteq A^2(\C^2,\mu_\Om).
\]
Finally, to show that the inclusion is strict, let us consider $\widetilde{G}(z_1,z_2)=\sum\limits_{(m_1,m_2)\in\N^2}\widetilde{\ell}_{m_1,m_2}z_1^{m_1}z_2^{m_2}$, where
    \beas
    \widetilde{\ell}_{m_1,m_2}=\begin{cases}
        \frac{2^{2k}}{k^{\frac{3}{4}} (4k+2)!^{\frac{1}{2}}},& \text{when } m_1=m_2=k>0,\\
          0,&\text{otherwise}.
      \end{cases} 
    \eeas
    Then $G\in\hol(\C^2)$, and by \eqref{eq:norm eq PW}, it follows that
    $\|\widetilde{G}\|_{A^2(\C^2,\mu_\Om)}^2\approx \sum\limits_{k=1}^\infty \frac{1}{k^\frac{3}{2}}<\infty.$ Hence, $\widetilde{G}\in A^2(\C^2,\mu_\Om)$.
    However, by Stirling's approximation, it follows that
  \begin{align*}
      \|\widetilde{G}\|_{A^2(\C^2,\omega_\Om)}^2\approx \sum\limits_{k=1}^\infty \frac{1}{k}=\infty.
  \end{align*}
  Hence, $\widetilde{G}\notin A^2(\C^2,\omega_\Om)$, which completes the proof of the theorem.}
  \end{section}
 \begin{section}{Further remarks}
 In this section, we expand on the scope of the geometric conditions presented in our theorems. We begin by giving an example of a class of bounded $\cont^2$-smooth strongly convex domains, whose dual complement may not be strongly convex.

 Let us consider the following class of real ellipsoids in $\C^n$.
\bea\label{eq:counter nD}
\Om_n=\left\{z=(z_1,z_2,\cdots,z_n)\in\C^n: \sum_{j=1}^n a_j(\operatorname{Re}z_j)^2 +\sum_{j=1}^n b_j(\operatorname{Im}z_j)^2<1\right\},
\eea
where $a_j$ and $b_j$ are positive constants such that there exists at least one index $j\in\{1,2,\cdots,n\}$, with $a_j\neq b_j$.
 \begin{lemma}\label{le:counter ellip}
 Let $\Om_n$ be a real ellipsoid as defined in \eqref{eq:counter nD}, where $a_j \neq b_j$, for some $j \in \{1, 2, \dots, n\}$. Then its dual complement $\Om_n^*$ is not strongly convex when $a_j = 2b_j$ or $b_j = 2a_j$, and it fails to be even convex when $\max\left\{\frac{a_j}{b_j}, \frac{b_j}{a_j}\right\}>2$.  
 \end{lemma}
 \begin{proof}
When $n=1$, by the definition of the dual complement it follows that
 $\zeta\in \Om_1^*\setminus\{0\}$ if and only if $\zeta^{-1}\in \overline{\Om_1}$. Thus, applying  the transformation $\zeta\rightarrow \zeta^{-1}$ on $\overline{\Om_1}\setminus\{0\}$, we can verify that
\[
\Om_1^*=\{z\in\C: ((\rea z)^2+(\Ima z)^2)^2< a_1(\rea z)^2+b_1(\Ima z)^2\}.
\]
    It is straightforward to see that $\Om_1^*$ is a bounded $\cont^2$-smooth domain. By analyzing the hessian of the defining function we get that $\Om_1^*$ is convex but not strongly convex when $2a_1=b_1$ or $2b_1=a_1$, and is not convex when $\max\{\frac{a_1}{b_1},\frac{b_1}{a_1}\}> 2$.

When $n\geq2$, it is difficult to explicitly compute the dual complement $\Om_n^*$. However, we can determine the failure of convexity or strong convexity by analyzing its linear slices. Without loss of generality, let us assume $a_1\neq b_1$ and consider the complex line $\ell_{z_1} =\{ (z_1, 0, \dots, 0) : z_1 \in \C\}$. The slice $\Om_n^*\cap \ell_{z_1}$ is a planar domain that is identical to $\Om_1^*$. To see this, note that a point $w = (w_1, 0, \dots, 0)$ belongs to $\Om_n^*$ if and only if $w_1 z_1 \neq 1$ for all $z \in \Om_n$. If $w \in \Om_n^*$ and $z_1 \in \Om_1$, then the point $(z_1, 0, \dots, 0)$ lies in $\Om_n$. Therefore, $w_1 z_1 \neq 1$ for any $z_1 \in \Om_1$, which implies $w_1 \in \Om_1^*$. Conversely, suppose $w_1 \in \Om_1^*$ and $z = (z_1, z_2,\dots,z_n) \in\Om_n$. Then $z_1\in\Om_1$, which guarantees that the pairing $\sum_{j=1}^n w_j z_j = w_1 z_1 \neq 1$. This implies $(w_1, 0, \dots, 0) \in \Om_n^* \cap \ell_{z_1}$.
Finally, based on the convexity properties of $\Om_1^*$, it follows that the slice $\Om_n^* \cap \ell_{z_1}$, and consequently $\Om_n^*$ is not strongly convex when $a_1 = 2b_1$ or $b_1 = 2a_1$, and fails to be convex when $\max\left\{\frac{a_1}{b_1}, \frac{b_1}{a_1}\right\}>2$, which completes the proof.
 \end{proof}
\begin{remark}\label{rk:thm 1 2}
As mentioned in the introduction, we are unaware of any general conditions on bounded $\cont^2$-smooth strongly $\C$-convex domains such that their dual complements are strongly convex. However, we can construct a class of domains satisfying this property. Let $E$ be a bounded $\cont^2$-smooth strongly convex domain containing the origin, and $\Om= E^*$. Then, $\Om$ is a strongly $\C$-convex domain, and due to $\C$-convexity, $\Om^*= E$, which is strongly convex. This shows that the hypothesis of Theorem~\ref{th:main1} is satisfied for the dual complement of bounded $\cont^2$-smooth strongly convex domains containing the origin. For example, $\Om_n^*$, where $\Om_n$ is as defined in \eqref{eq:counter nD}, satisfies the hypothesis of Theorem~\ref{th:main1}.
\end{remark}
\begin{remark}\label{rk:ques L^p bdd}
    The hypothesis in Theorem~\ref{th:main1} of $\Om^*$ being strongly convex is used only to show that $\mathcal F_{n+1}$ is a $L^p$-bounded operator between $A^p(\Om)$ and $A^p(\Om^*)$, which essentially follows from the fact that when $D$ is a bounded $\cont^2$-smooth strongly convex domain containing the origin, the operator $|\mathcal B_{D}|$, as given in Lemma~\ref{le:bdd of LS ker} is bounded on $L^p(D)$. Consequently, the requirement of $\Om^*$ being strongly convex can be relaxed if one can answer the following question.
    
\noindent\textbf{Question.} Let $D\subset\C^n$ be a bounded $\cont^2$-smooth strongly $\C$-convex domain containing the origin, and $|\mathcal B_D|$ is the operator as given in Lemma~\ref{le:bdd of LS ker}. Is $|\mathcal {B} _ {D}|$ a $L^p(D)$-bounded operator?
\end{remark}
\begin{remark}\label{rk:C^1 smooth}
   Note that the domain considered in Theorem~\ref{th:counter}  is a bounded convex Reinhardt domain which lacks $\cont^1$-smoothness, which raises the following question.
\newline
\noindent\textbf{Question.} Does there exist a bounded $\cont^1$-smooth convex domain $\Om\subset \C^n$ such that $\mathcal F_{n+1}$ is not a normed space isomorphism from $A^2(\Om)$ and $A^2(\Om^*)$, and $\mathcal L$ is not a normed space isomorphism from $A^2(\Om)$ to both $A^2(\C^n,\omega_\Om)$ and $A^2(\C^n,\mu_\Om)$?
\end{remark}
 \end{section}
\bibliography{laplaceleray}{}
\bibliographystyle{plain}
\end{document}